\documentclass[final,20pt]{elsarticle}

\usepackage{amssymb}
\usepackage{algorithmic}
\usepackage{algorithm}
\usepackage[all,pdf]{xy}
\usepackage{amsfonts}
\usepackage{lscape}
\usepackage{color}
\usepackage{bbm}
\usepackage{amsfonts,amssymb,mathrsfs,amsmath,amssymb,theorem,float}
\usepackage{graphicx}
\usepackage{lineno}
\usepackage{hyperref}
\hypersetup{hypertex=ture,
colorlinks=true,
linkcolor=blue,
anchorcolor=blue,
citecolor=blue}

\newtheorem{theorem}{Theorem}[section]

\newtheorem{lemma}[theorem]{Lemma}

\newtheorem{definition}[theorem]{Definition}
\newtheorem{example}[theorem]{Example}

\def\N{{\mathbb N}}
\def\Z{{\mathbb Z}}

\def\X{{\mathbb{X}}}
\def\F{{\mathbb{F}}}
\def\H{{\mathbb{H}}}

\def\inp{{\rm{in}}}
\def\out{{\rm{o}}}

\def\res{\hbox{\rm{res}}}
\def\LC{\hbox{\rm{LC}}}

\def\MoCont{\hbox{\rm{MoCont}}}
\def\MoPrim{\hbox{\rm{MoPrim}}}
\def\Cont{\hbox{\rm{Cont}}}
\def\Prim{\hbox{\rm{Prim}}}

\def\qed{\hfil {\vrule height5pt width2pt depth2pt}}

\journal{Journal of XXXX}

\begin{document}

\begin{frontmatter}



\title{A New Sparse Algorithm for Polynomial GCD over Integers}

\author[firstaddress]{Qiao-Long Huang\corref{mycorrespondingauthor}}
\cortext[mycorrespondingauthor]{Corresponding author}
\ead{huangqiaolong@sdu.edu.cn}

\author[secondaddress]{Michael Monagan}
\ead{mmonagan@sfu.ca}

\address[firstaddress]{School of Mathematics, Shandong University, Jinan, {\rm250100}, China}
\address[secondaddress]{Department of Mathematics, Simon Fraser University, Burnaby, British Columbia, V5A 1S6, Canada}



\begin{abstract}
We describe a new greatest common divisor (GCD) algorithm for polynomials with integer coefficients.
The bit complexity of the new algorithm is polynomial in the input and
output sizes and the individual degree bounds.
%
 Our algorithm reduces a multivariate polynomial GCD to a single univariate polynomial GCD.
The main idea of our algorithm is a new variable
substitution which reduces a multivariate polynomial to
a separated one, that is, the coefficients in a main variable are all monomials. The explicit bit complexity is analyzed and we have implemented our algorithm in Maple. It is shown that our algorithm is
efficient for polynomials with high degree, large number of variables, but small number of terms in GCD.
\end{abstract}




\end{frontmatter}


\section{Introduction}
Multivariate polynomial GCD computation is one of the central problems in algebraic
and symbolic computation.
It plays an important role as a subroutine in many
other problems, for example, the solution of polynomial equations,
symbolic integration, simplification, polynomial factorization \cite{HuangG23,ChenM23,MonaganTuncerT20,MonaganPaluck22}, secret sharing \cite{CafaroP18,VijayakumarBK13}, cryptography \cite{CoppersmithFPR96} and so on.

In 1967, Collins \cite{Collins67} gave the first modern GCD algorithm by refining the Euclidean algorithm.
To get a fraction free $Euclidean$ $polynomial$ 
$remainder$ $sequence$ (PRS), pseudo-division is used during the computation. This leads to the so-called intermediate expression swell phenomenon, which makes Collins's algorithm lack scalability.

In 1971, Brown \cite{brown1971euclid} gave a more efficient  algorithm based on interpolation for dense polynomials.
Brown's algorithm solves the intermediate expression swell problem by introducing the modular method,
that is, by substituting all but one of the variables by certain integers,
multivariate polynomial GCD computation becomes univariate GCD computation,
and the true GCD will be recovered from these univariate GCDs either by
interpolations or by the Chinese Remainder Theorem.
Brown's algorithm works effectively for dense multivariate polynomial GCD computation. For the sparse case, the Extended Zassenhaus algorithm (EZ-GCD) is better suited. Moses and Yun \cite{MosesY73} proposed this algorithm in 1973, using Hensel lifting instead of interpolation to recover the GCD.
%
%
%
In 1980, Wang \cite{Wang80b} proposed an enhanced EZ-GCD algorithm, called the EEZ-GCD algorithm, which improved the EZ-GCD algorithm by solving the leading coefficient problem, bad-zero problem, unlucky evaluation problem, and the common divisor problem.

In 1979, Zippel \cite{Zippel79} developed the first modular sparse GCD algorithm based on sparse polynomial interpolation, which interpolates the GCD one variable at a time.
Zippel's algorithm is probabilistic, and its correctness relies on the Schwartz-Zippel \cite{ZIPPEL1990375} lemma. Zippel's algorithm is currently used as the main GCD algorithm in the computer algebra systems Fermat, Magma, Maple and Mathematica.

For the interested reader, \cite{Ben-OrT88,huang2021sparse2,arnold2016sparse,KaltofenLL00,javadi2010parallel,cuyt2008new,DBLP:conf/issac/Huang19,DBLP:journals/jsc/HuangG20,DBLP:journals/jsc/Huang23} provide more descriptions of sparse polynomial interpolation.

In 1988, Kaltofen \cite{kaltofen1988greatest} gave a GCD algorithm for polynomials given by straight-line programs.
This algorithm outputs the GCD in the form of straight-line program.
In 1990, Kaltofen  and Trager \cite{kaltofen1990computing} gave a GCD algorithm for polynomials given by black boxes.

In 1992, Sasaki and Suzuki \cite{SasakiS92} introduced the method of truncated power series, to reduce the intermediate expression swell of the Euclidean algorithm. This algorithm is efficient for calculating low-degree GCD of high-degree non-sparse polynomials.


In 2008, Cuyt and Lee \cite{cuyt2008new} proposed another improved technique.
Their algorithm is designed for the interpolation of a rational function, but a direct application is a GCD algorithm. 

In 2018, Tang, Li and Zeng \cite{tang2018computing} proposed
two GCD algorithms based on variations of Zippel's method and Ben-Or/Tiwari's interpolation algorithm \cite{Ben-OrT88}, respectively.

In 2021, Hu and Monagan \cite{hu2021fast} presented a parallel GCD algorithm for sparse GCD computation, which combined a Kronecker substitution with a Ben-Or/Tiwari \cite{Ben-OrT88} sparse interpolation modulo a smooth prime to determine the support of the GCD.

Although the current GCD algorithms used in software are effective in most cases, after all these years of development, explicit bit complexities for sparse multivariate polynomial GCD algorithms over integers seem not to be given.  One reason for this is that existing algorithms may construct an intermediate polynomial which is much larger than the input polynomials and their gcd. In previous work, the bit complexities were either mentioned to be polynomial in the number of variables, degrees, and the number of terms of the input and output polynomials~\cite{kaltofen1988greatest,kaltofen1990computing,tang2018computing,hu2021fast}
or given under certain conditions \cite{MosesY73,Wang80b,Zippel79}.
In this paper, we will give a new GCD algorithm and its exact bit complexity,
which is sensitive to sparse representations.
The given algorithm is shown to have better complexity than existing algorithms in many cases.
In practical computations, our algorithm performs better than the algorithm in Maple when the number of terms in GCD is small.

\subsection{Main results}
Let $\Z$ be the integer domain.
In this paper, we focus on GCD computation of polynomials over integers. Let $A$ and $B$ be two polynomials in $\Z[x_1,\dots,x_n]$ and $G=\gcd(A,B)$.  Let $D$ be a total degree bound for $A$ and $B$, that is, $D \ge \max(\deg(A),\deg(B))$.
 %
Let $\|A\|_{\infty}$ denote the largest integer coefficient of $A$ in magnitude.
The algorithm is randomized, so we assume that we can obtain a random bit with bit-cost $O(1)$.
In the following, all complexity analysis relies on the ``Soft-Oh" notation $O^\sim(\phi)=O(\phi\cdot \mathbf{polylog}(\phi))$, where $\mathbf{polylog}$ means $\log^c$ for some fixed $c>0$.

The main result of the paper is given below.
\begin{theorem}
Let $A,B$ be polynomials  in $\Z[x_1,\dots,x_n]$ with partial degree bound $d=\max_{i=1}^n \max(\deg_{x_i}A,\deg_{x_i}B)$ and $0<\varepsilon<1$.
Then there exists a randomized algorithm that takes as inputs $A,B$ and
returns $G=\gcd(A,B)$ with probability $\geq 1-\varepsilon$, using expected $O^\sim(n^2d^2T_{\out}^2\log^2 \frac{1}{\varepsilon}\log C+n^2dT_{\inp}\log^2 T_{\out}\log^2\frac{1}{\varepsilon}\log C)$ bit operations, where $T_{\inp}=\#A+\#B,T_{\out}=\#G$ and $\max\{\|A\|_{\infty},\|B\|_{\infty}\}\leq C$.
\end{theorem}

The algorithm is implemented in Maple and extensive numerical experiments
show that the new algorithm outperforms the default GCD algorithm in Maple in many cases,
especially when measured with the numbers of variables and total degrees. But as the number of terms in GCD increases, it becomes much slower than the Maple algorithm.
Details of the experiments can be found in section \ref{sec-exp}.

At  a high level, the algorithm is a combination of the modular method and a randomized Kronecker substitution.   It reduces a multivariate GCD computation to a single univariate GCD computation in $\Z[y]$.
The main ingredients of the algorithm include:
a new variable substitution is introduced to isolate all the terms of the GCD, that is, the coefficients of substituted GCD are monomials;
a new type of evaluation point is introduced to recover the GCD from its modular images
by using only small primes; a randomized testing algorithm is proposed to test if a polynomial is the real GCD.
Combination of these ingredients leads to the lower bit complexity
and the practical efficiency of the algorithm in sparse case.

\subsection{Related work and comparison}

In Table \ref{tab-1}, we list the complexities for the GCD algorithms which have explicit analysis in the literature.
In the table, ``Total Cost" refers to the number of bit operations.
%
%
Compared with \cite{BrownT71} and \cite{Collins67}, if $A,B$ and $G$ are sparse, our method has better ``Soft-Oh" complexity. However, for the dense case, $T_{\inp}$ and $T_{\out}$ are $O(d^n)$, our complexity has no advantage.

\begin{table}[H]
\footnotesize\tabcolsep 6pt
\begin{center}
\caption{A ``soft-Oh" comparison  for GCD algorithms over $\Z[x_1,\dots,x_n]$}\label{tab-1}\vspace{-2mm}
\end{center}
\begin{center}
\setlength{\tabcolsep}{1.5mm}{
\begin{tabular}{c|c|c}
Algorithms &Total Cost & Type\\ \cline{1-3}
Brown \cite{BrownT71}& $d^n\log^2C +d^{n+1}\log C$  & Las Vegas\\
Collins \cite[analyzed in \cite{BrownT71}]{Collins67} &$d^{4n}2^{2n^2}3^n\log^2 C$&  Deterministic\\
This paper& $n^2d^2T_{\out}^2\log^2 \frac{1}{\varepsilon}\log C+n^2dT_{\inp}\log^2 T_{\out}\log^2\frac{1}{\varepsilon}\log C$ &Monte Carlo\\
\end{tabular}}
\end{center}
\end{table}

 Moses and Yun \cite{MosesY73} gave the EZ (Extended Zassenhaus) GCD algorithm, which is a direct extension of Hensel's lemma. This algorithm
seems to have a computational bound. In most cases, it is a polynomial function of $T_{\inp},d$ and $n$.
However, some conditions may increase the complexity. For example, the number of terms in the expanded series form
of $A(x_1,x_2-a_2,\dots,x_n-a_n)$ have larger order than that of in $A(x_1,x_2,\dots,x_n)$ for some $(a_2,\dots,a_n)$; or the polynomial is not monic with respect to any variable; or the GCD has a different common divisor with each of the cofactors.
Our algorithm works for any polynomials and has no extra assumption in the input.

 Zippel \cite{Zippel79} introduced sparse interpolation into Brown's interpolation scheme. His technique preserves sparsity of the multivariate GCD during the computation. It should be noted that Zippel's algorithm is not polynomial time in the sparse representation. This is because
his algorithm computes the content and the primitive part of the GCD separately.
However, some sparse polynomials have dense primitive
parts (cf. \cite[Section 5]{GathenK85}).

In order to avoid the content computation, Kaltofen  \cite{kaltofen1988greatest} reduced multivariate polynomials to monic bivariate polynomials, and obtained the first polynomial time algorithm in the sparse representation and the partial
degree bound.
The degree of the reduced bivariate polynomial is $O(D)$, so the number of terms is $O(D^2)$.
Since the height of coefficients is $O(nd\log C)$, it needs $O(ndD^2\log C)$ bits to represent this polynomial. 
Because $D\in O(nd)$, the complexity is at least $O(n^3d^3\log C)$. As both degrees of $n$ and $d$ in the complexity are $2$, our algorithm is better than his algorithm in $n$ and $d$.

Wang \cite{Wang80b} described an enhanced gcd algorithm based on the EZ-GCD algorithm, which is called EEZ-GCD algorithm. This algorithm improves the EZ-GCD algorithm by solving the leading coefficient problem, bad-zero problem, unlcuky evaluation problem and the common divisor problem. In the EEZ-GCD algorithm, an important step
is to predetermine the leading coefficient of the
desired GCD, which needs to compute the factorization of the leading coefficient. However, this step may lead a large complexity because the number of terms of the factors may be very large.
Our algorithm avoids this problem by isolating the leading coefficient into a single term to make it easy to recover.

Tang, Li and Zeng \cite{tang2018computing} described
two methods via multivariate polynomial interpolation, based on a variation of Zippel's
method and the Ben-Or/Tiwari algorithm, respectively.
Their algorithms construct the auxiliary functions by introducing a homogenizing variable, and then recover the polynomial coefficients
of the homogenizing variable via sparse interpolation algorithms, and finally recover the GCD.
For the algorithm based on Zippel's algorithm, they combine all exponents of $x_i$ in each monomial to form all candidate monomials with a certain total degree, which may lead to a large number of candidates.
For the algorithm based on Ben-Or/Tiwari over integers, they reduce $f$ to polynomials $f(p^j_1y+\sigma_1,\dots,p^j_ny+\sigma_n)$, $j=1,2,\dots,2T_{\out}$, where the $p_i$'s are different primes and the $\sigma_i$'s are constants. This kind of substitution was first proposed by  Cuyt and Lee \cite{cuyt2008new}.
The degree of these univariate polynomials is $O(D)$, and the height of the coefficients is $O(DT_{\out})$. Therefore, only computing these reductions requires at least  $O^\sim(D^2T_{\inp}T^2_{\out})$ bit operations, that is $O^\sim(n^2d^2T_{\inp}T^2_{\out})$ bit operations. Compared with their algorithm, our algorithm is better in $T_{\inp}$ and $T_{\out}$.
 As they said in their paper, their computational complexity can be reduced if the modular arithmetic technique is employed. According to our rough analysis, a prime $p\in O^\sim(n^D)$ should be selected, which may require $O(\log^3 p)$ bit operations \cite[The. 18.8]{GaGer13}, that is, $O(n^3d^3)$ bit operations. Compared with the modular version of this algorithm, our algorithm is better in $n$ and $d$ as the degrees of $n$ and $d$ in the complexity are $2$.

Hu and Monagan \cite{hu2021fast} presented a parallel GCD algorithm for sparse multivariate polynomials with
integer coefficients. The algorithm combines a Kronecker substitution with a Ben-Or/Tiwari sparse interpolation modulo a smooth prime.  Compared with their algorithm, we use randomized Kronecker substitution and only use primes with small sizes.
Moreover, their algorithm first computes $\Delta\cdot \gcd(A,B)$ instead of $\gcd(A,B)$ for some $\Delta\in\Z[x_1,\dots,x_{n-1}]$. Since $\#(\Delta\cdot \gcd(A,B))$ may be much larger than $\#(G)$, their algorithm may have large complexity.
Compared with their algorithm, we control the size of $\#(\Delta)$ by isolating the leading coefficient, so that $\#(\Delta)=1$. This difference causes us to analyze an exact complexity, while their algorithm cannot predict the size of $\#(\Delta)$ in theory.

Our work is inspired by the work of Tang, Li and Zeng \cite{tang2018computing} and Hu and Monagan \cite{hu2021fast}, where the basic idea is to reduce a multivariate GCD to a univariate GCD.
The new approach is to apply the following reduction $A(p_1y^{s_1},\dots,p_ny^{s_n})$ where the $s_i \in \N$ are randomly chosen and the $p_i$ are randomly chosen primes.
Compared with the algorithms of Tang, Li and Zeng~\cite{tang2018computing}, our algorithm avoids the shift of polynomials that may require a lot of computation, but increases the degree of $y$.
If $s_1=1,\dots,s_n=d^{n-1}$, then it is the classical Kronecker substitution, which has been used in Hu and Monagan's algorithm~\cite{hu2021fast}. If the GCD is sparse, our randomized Kronecker substitution has a lower degree.
If $s_i$'s are selected well enough, $G(p_1y^{s_1},\dots,p_ny^{s_n})$ and $G$ will
have the same number of terms, which means that all terms in $G$ have different degrees in $y$ after substitution, so $G$ can be recovered by dividing the integer coefficients of $G(p_1y^{s_1},\dots,p_ny^{s_n})$ by the primes $p_i.$






\section{Definition and Notation}
\subsection{Notations}
Let $A=c_1M_1+\cdots+c_tM_t\in\mathcal{R}[x_1,\dots,x_n]$, where $\mathcal{R} $ is a unique factorization domain (UFD), $c_i\in\mathcal{R}(\neq 0)$ and $M_i,i=1,\dots,t$ are monomials. Assume the exponent vector of $M_i$ is $\mathbf{e}_i=(e_{i,1},\dots,e_{i,n})$ and the monomials $M_i,i=1,\dots,t$ are arranged in lexicographically increasing order of $\mathbf{e}_i,i=1,\dots,t$. Here, we always assume that the lexicographical order is generated by $x_n>\cdots >x_1>y$. $c_tM_t$ is called the {\em leading term}. $c_t$ is the {\em leading coefficient}, and is denoted by $\LC(A)$.
Definition \ref{def-0} defines a GCD for a UFD.  This paper involves three UFDs: the integers $\Z$, a polynomial ring over finite field $\F_q[x_1,\dots,x_n]$ and the polynomial ring $\Z[x_1,\dots,x_n]$.

\begin{definition} \label{def-0}
Let $A,B$ be given nonzero elements of a UFD. Then $G$ is called the greatest common divisor (GCD) of $A,B$ if
\begin{enumerate}
\item $G$ divides $A,B$,
\item every common divisor of $A,B$ divides $G$, and
\item $G$ is $unit\ normal$.
\end{enumerate}
Condition 3 imposes uniqueness on $G$.
For integers $\Z$, $unit\ normal$ means $G$ is positive.
For polynomials in $\F_q[\X]$, $unit\ normal$ means  $\LC(G)$ is $1$.
For polynomials in $\Z[x_1,\dots,x_n]$, $unit\ normal$ means $\LC(G)$ is positive.
We say that $A$ is similar to $B$, $A\sim B$, if there exists a unit $a\in \mathcal{R}^*$ such that $a A=B$.
\end{definition}

\begin{definition}\label{def-3}
A term $cM$ is called the {\em monomial content} of a polynomial $f=c_1M_1+\cdots+c_tM_t\in \Z[x_1,\dots,x_n]$, if
 $c=\gcd(c_1,\dots,c_t)$ and $M=\gcd(M_1,\dots,M_t)$, where $M_i,i=1,\dots,t$ are different monomials and $c_i,i=1,\dots,t$ are non-zero coefficients.
Here $c$ is the content of $f$. Denote $c=\Cont(f)$ and denote the primitive part $f/c$ as $\Prim(f)$.
Denote $cM$ by $\mathrm{MoCont}(f)$ and we call $f/\mathrm{MoCont}(f)$ the {\em monomial primitive part} of $f$ and denote it by $\mathrm{MoPrim}(f)$.
\end{definition}

\begin{definition}\label{def-1}
We call a polynomial $f\in \Z[x_1,\dots,x_n,y]$ {\em separable} w.r.t. $y$ if $f$ has the form
$$f=c_1M_1y^{d_1}+c_2M_2y^{d_2}+\cdots+c_tM_ty^{d_t},$$
where $c_i(\neq 0)\in \Z$, $M_i,i=1,\dots,t$ are monomials in $\Z[x_1,\dots,x_n]$ and $d_i,i=1,\dots,t$ are different exponents.
\end{definition}
In other words, $f$ is separable w.r.t. $y$ if each coefficient of $f$ w.r.t $y$ has only one term.
\begin{definition}
Let $\mathbf{s}=(s_1,\dots,s_n)\in\N^n$ be an integer vector and $A\in\Z[x_1,\dots,x_n]$.
Define
\begin{equation}\label{eq-1}
A_{(\mathbf{s},y)}:=\frac{A(x_1y^{s_1},\dots,x_ny^{s_n})}{y^k}
\end{equation}
where $k$ is the smallest exponent of $y$ in $A(x_1y^{s_1},\dots,x_ny^{s_n})$.
\end{definition}

Assume $f(\X,y)\in\Z[x_1,\dots,x_n,y]$ and for $\mathbf{a}=(a_1,\dots,a_n)$ denote $f(\mathbf{a},y):=f(a_1,\dots,a_n,y)$. For a prime $p$, let $\Phi_p$ denote the modular mapping
\begin{equation}\label{eq-2}
\Phi_p(f):=f~{\rm mod}~p.
\end{equation}

\subsection{Preliminary results}
First, observe that
if $f(=\mathrm{MoPrim}(f))$ is a monomial primitive polynomial in $\Z[x_1,\dots,x_n]$, so are its factors.
Denote $y^\mathbf{s}\X:=(x_1y^{s_1},\dots,x_ny^{s_n})$ and $y^{-\mathbf{s}}\X:=(x_1/y^{s_1},\dots,x_n/y^{s_n})$, where $\mathbf{s}=(s_1,\dots,s_n)\in \N^n$.

\begin{lemma}\label{the-5}
Let $A,B\in \Z[x_1,\dots,x_n]$, $G=\gcd(A,B)$, and $\mathbf{s}=(s_1,\dots,s_n)\in \N^n$.
Then $\gcd(A_{(\mathbf{s},y)},B_{(\mathbf{s},y)})= G_{(\mathbf{s},y)}$.
\end{lemma}
\begin{proof}
We first claim that $\gcd(A(y^\mathbf{s}\X),B(y^\mathbf{s}\X))\sim y^mG(y^\mathbf{s}\X)$ for some integer $m\geq 0.$
Proof of the claim:
$G|A$ and $G|B$ imply that $G(y^\mathbf{s}\X)|A(y^\mathbf{s}\X)$ and $G(y^\mathbf{s}\X)|B(y^\mathbf{s}\X)$. Let $P=\gcd(A(y^\mathbf{s}\X),B(y^\mathbf{s}\X))$, then we have $G(y^\mathbf{s}\X)|P$.
We prove the reverse direction. Since $P|A(y^\mathbf{s}\X)$,
there exists a $Q\in\Z[x_1,\dots,x_n,y]$ that makes $A(y^\mathbf{s}\X)=P(\X, y)Q(\X,y)$. Replacing $x_iy^{s_i}$ by $x_i$, we have
$A=P(y^{-\mathbf{s}}\X,y)$ $Q(y^{-\mathbf{s}}\X,y)$. There exists an integer $k_1$ such that $Q(y^{-\mathbf{s}}\X,y)=y^{k_1}(Q_{\ell}y^{d_{\ell}}+\cdots+Q_1y^{d_1}+Q_0)$, where $Q_i\in \Z[x_1,\dots,x_n]$ and $d_i>0$. So $y^{k_1}P(y^{-\mathbf{s}}\X,y)$ is a polynomial in $\Z[x_1,\dots,x_n,y]$ and $y^{k_1}P(y^{-\mathbf{s}}\X,y)|A$.
For the same reason, there exists an integer $k_2$ such that $y^{k_2}P(y^{\mathbf{s}}\X,y)|B$.
Without loss of generality, assume $k_1\geq k_2$. Then $y^{k_1}P(y^{-\mathbf{s}}\X,y)|A$ and $y^{k_1}P(y^{-\mathbf{s}}\X,y)|y^{k_1-k_2}B$.
So $y^{k_1}P(y^{-\mathbf{s}}\X,y)|\gcd(A,y^{k_1-k_2}B)$.
As $\gcd(A,y^{k_1-k_2}B)=\gcd(A,B)=G$,
$y^{k_1}P(y^{-\mathbf{s}}\X,y)|G$. Replacing $x_i/y^{s_i}$ by $x_i$, we have $y^{k_1}P(\X,y)|G(y^{\mathbf{s}}\X)$.
If $k_1\geq 0$, then $P(\X,y)|G(y^{\mathbf{s}}\X)$.
If $k_1<0$, let $m'=-k_1$, then
$P(\X,y)|y^{m'}G(y^{\mathbf{s}}\X)$.
So there exists an integer $m\geq 0$ such that $P\sim y^{m}G(y^{\mathbf{s}}\X)$.
The claim is proved.

Let $C=\gcd(A_{(\mathbf{s},y)},B_{(\mathbf{s},y)})$. Since $A_{(\mathbf{s},y)}=\frac{A(y^\mathbf{s}\X)}{y^{d_A}}$, $Cy^{d_A}|A(y^\mathbf{s}\X)$.
Here $d_A$ is the integer $k$ in Eq. (\ref{eq-1}).
For the same reason, $Cy^{d_B}|B(y^\mathbf{s}\X)$.
Then $Cy^{\min\{d_A,d_B\}}|\gcd(A(y^\mathbf{s}\X),B(y^\mathbf{s}\X))$.
By the claim, $$\gcd(A(y^\mathbf{s}\X),B(y^\mathbf{s}\X))\sim y^mG(y^\mathbf{s}\X)$$ for some integer $m\geq 0$,
so $Cy^{\min\{d_A,d_B\}}|y^mG(y^\mathbf{s}\X)$.
Thus
$C|\frac{y^mG(y^\mathbf{s}\X)}{y^{\min\{d_A,d_B\}}}.$
Clearly, $d_G\leq d_A$ and $d_G\leq d_B$. So
$d_G\leq \min\{d_A,d_B\}$ and  $C|\frac{y^mG(y^\mathbf{s}\X)}{y^{d_G}}=y^mG_{(\mathbf{s},y)}$.
For the reverse direction, since $G|A$ and $G|B$, we have $G(y^\mathbf{s}\X)|A(y^\mathbf{s}\X)$ and $G(y^\mathbf{s}\X)|B(y^\mathbf{s}\X)$.
Since $G(y^\mathbf{s}\X)=G_{(\mathbf{s},y)}\cdot y^{d_G}$ and $A(y^\mathbf{s}\X)=A_{(\mathbf{s},y)}\cdot y^{d_A}$,
$G_{(\mathbf{s},y)}|A_{(\mathbf{s},y)}$. For the same reason, we have $G_{(\mathbf{s},y)}|B_{(\mathbf{s},y)}$.
So we have $G_{(\mathbf{s},y)}|\gcd(A_{(\mathbf{s},y)},B_{(\mathbf{s},y)})=C.$
So there exists an integer $m'$ such that $C\sim y^{m'}G_{(\mathbf{s},y)}$. Regarding $C$ and $G_{(\mathbf{s},y)}$ as polynomials in $y$ with coefficients in $\Z[x_1,\dots,x_n]$, we know that both $C$ and $G_{(\mathbf{s},y)}$ have non-zero constants, so $m'=0$. As $\Z$ has only units $\pm 1$ and $G_{(\mathbf{s},y)}$ has the same leading coefficient as $G$, $\gcd(A_{(\mathbf{s},y)},B_{(\mathbf{s},y)})=G_{(\mathbf{s},y)}$. The lemma is proved.
\end{proof}\qed

\begin{theorem}\label{the-2}
Assume $f=c_1M_1y^{d_1}+\cdots+c_tM_ty^{d_t}\in \Z[x_1,\dots,x_n,y]$ is separable w.r.t. $y$, where $M_i,i=1,\dots,t$ are monomials and $d_i,i=1,\dots,t$ are different exponents of $y$. Let $p_1,\dots,p_n$ be any different primes with $p_i\nmid c_j$ for all $i=1,\dots,n,j=1,\dots,t$. If $\mathrm{MoCont}(f)=1$, then
the content of $f(p_1,\dots,p_n,y)$ is $1$.
\end{theorem}

\begin{proof}
$f(p_1,\dots,p_n,y)=c_1M_1(p_1,\dots,p_n)y^{d_1}+\cdots+c_tM_t(p_1,\dots,p_n)y^{d_t}$.
So the content of $f(p_1,\dots,p_n,y)$ is the GCD of $c_iM_i(p_1,\dots,p_n),i=1,\dots,t$. Since $\mathrm{MoCon}(f)=1$ and $p_i\nmid c_j$, the content of $f(p_1,\dots,p_n,y)$ is $1$.
\end{proof}\qed

Our proof will make extensive use of the Schwartz-Zippel Lemma.
\begin{lemma}\label{lm-6}\cite[Lemma 6.44]{GaGer13} Let $\mathcal{R}$ be an integral domain  and $A \in \mathcal{R}[x_1,\dots,x_n]$
be non-zero and with total degree $D$ and let $S\subset \mathcal{R}$ be a finite set. If $\mathbf{a}=(a_1,\dots,a_n)$ is chosen at random from $S^n$ then ${\rm Prob}[A(\mathbf{a})=0]\leq \frac{D}{|S|}$.
\end{lemma}

We also need the following variation of Schwartz-Zippel Lemma.

\begin{lemma}\label{lm-8}
Let $\mathcal{R}$ be an integral domain and $A \in \mathcal{R}[\X]$
be non-zero and with total degree $D$ and let $S\subset \mathcal{R}$ be a finite set with $|S|\geq n$. If $\mathbf{a}=(a_1,\dots,a_n)$ is chosen at random from $S^n$ and $a_i\neq a_j$ for $i\neq j$, then ${\rm Prob}[A(\mathbf{a})=0]\leq \frac{D}{|S|-n+1}$.
\end{lemma}
\begin{proof}
Denote $N=|S|$. It suffices to show that $A$ has at most $DN(N-1)\cdots (N-n+2)$ zeros such that each entry is different from each other.
We prove it by induction on $n$. The case $n=1$ is clear, since a nonzero univariate polynomial of degree at most $D$ over an integral domain has at most $D$ zeros.
For the induction step, we write $A$ as a polynomial in $x_n$: $A=\sum_{0\leq i\leq k}A_ix_n^i$ with $A_i\in \mathcal{R}[x_1,x_2,\dots,x_{n-1}]$ for $0\leq i\leq k$ and $A_k\neq 0$. Then $\deg(A_k)\leq D-k$.
By the induction hypothesis, $A_k$ has at most $(D-k)N(N-1)\cdots (N-n+3)$ such zeros in $S^{n-1}$, where each of these entries is different from the others. So there are at most $(D-k)N(N-1)\cdots (N-n+3)(N-n+1)(\leq(D-k)N(N-1)\cdots(N-n+3)(N-n+2))$ common zeros of $A$ and $A_k$ in $S^n$ where each of these entries is different from the others.
Furthermore, for each $(a_1,\dots,a_{n-1})\in S^{n-1}$ with $A_k(a_1,\dots,a_{n-1})\neq 0$, the univariate polynomial $A(a_1,\dots,a_{n-1},x_n)=\sum_{0\leq i\leq k}A_i(a_1,\dots,a_{n-1})x_n^i\in \mathcal{R}[x_n]$ of degree $k$ has at most $k$ zeros, so there are at most $k N(N-1)\cdots (N-n+3)(N-n+2)$ zeros.
Overall, the total number of zeros of $A$ in $S^n$ is bounded by $(D-k)N(N-1)\cdots (N-n+3)(N-n+2)+k N(N-1)\cdots (N-n+3)(N-n+2)=D N(N-1)\cdots (N-n+3)(N-n+2)$.
As there are $N(N-1)\cdots (N-(n-1))$ points in $S^n$ where each of these entries is different from the others, ${\rm Prob}[A(\mathbf{a})=0]\leq \frac{D N(N-1)\cdots (N-n+3)(N-n+2)}{N(N-1)\cdots (N-(n-1))}=\frac{D}{N-n+1}$.
\end{proof}\qed

The following theorem shows that $f\in\mathbb{Z}[\X]$ can be separated by introducing a new variable $y$.

\begin{theorem}\label{the-1}
Assume $f\in \Z[x_1,\dots,x_n]$ and $T\geq \#f$. If $\mathbf{s}=(s_1,\dots,s_n)\in\N^n$ is chosen randomly from $[0,\frac92T(T-1)]^n$, then with probability $\geq 8/9$, $f(x_1y^{s_1},\dots,x_ny^{s_n})$ is separated w.r.t. $y$.
\end{theorem}

\begin{proof}
Denote $t:=\# f$. If $t=1$, it is obvious. Now assume $T\geq t\geq 2$. Assume $f=c_1x_1^{e_{1,1}}\cdots x_n^{e_{1,n}}+\cdots+c_tx_1^{e_{t,1}}\cdots x_n^{e_{t,n}}$. Then
$$f(x_1y^{s_1},\dots,x_ny^{s_n})=c_1x_1^{e_{1,1}}\cdots x_n^{e_{1,n}}y^{\sum_{i=1}^ns_ie_{1,i}}+\cdots+c_tx_1^{e_{t,1}}\cdots x_n^{e_{t,n}}y^{\sum_{i=1}^ns_ie_{t,i}}.$$
So $f(x_1y^{s_1},\dots,x_ny^{s_n})$ is separated w.r.t. $y$ if and only if all the exponents $S_i:=s_1e_{i,1}+\cdots+s_ne_{i,n},i=1,\dots,t$ of $y$ are different.
Let $P(s_1,\dots,s_n)=\prod_{i<j}(S_i-S_j)$ be a polynomial with degree $\frac12t(t-1)\leq \frac 12T(T-1)$.
So by Schwartz-Zippel lemma (Lemma~\ref{lm-6}), if a vector $(s_1,\dots,s_n)$ is randomly chosen from $[0,\frac92T(T-1)]^n$, the probability of $P(s_1,\dots,s_n)\neq 0$ is at least $1-\frac{\frac12T(T-1)}{\frac92T(T-1)}=8/9$.
\end{proof}\qed

\begin{example}
\label{ex-11}
Let $f=4x_1^2x_2^3+5x_1^5x_2^6+10x_1^4x_2^4\in\Z[x_1,x_2]$. For $\mathbf{s}\in \N^2$, the exponents of $y$ in $f(x_1y^{s_1},x_2y^{s_2})$ are
\[
\begin{cases}
S_1=2s_1+3s_2\\
S_2=5s_1+6s_2\\
S_3=4s_1+4s_2
\end{cases}
\]
If we choose $s_1:=2,s_2:=3$, then $S_1=13,S_2=28,S_3=20$, which makes $P(2,3)\neq 0$. Thus $f(x_1y^2,x_2y^3)$ is separated w.r.t $y$.
\end{example}

\subsection{Resultant}
Let $F_1=\sum_{i=0}^{d}a_iy^i$ and $F_2=\sum_{i=0}^{\ell} b_iy^i$ be non-zero. The $Sylvester\ matrix$ of $F_1,F_2$ is the $d+\ell$ by $d+\ell$ matrix

\begin{equation}
\left(
  \begin{array}{cccccccc}
    a_d & a_{d-1} & \cdots & a_1 & a_0 &  & \\
        & a_d & a_{d-1} & \cdots & a_1 & a_0 & \\
        &    & \cdots & \cdots & \cdots & \cdots &  \\
     &  &  & a_d & \cdots & \cdots & a_0  \\
    b_{\ell} & b_{\ell-1} & \cdots & b_1 & b_0 &  & \\
        & b_{\ell} & b_{\ell-1} & \cdots & b_1 & b_0 & \\
        &    & \cdots & \cdots & \cdots & \cdots &  \\
     &  &  & b_{\ell} & \cdots & \cdots & b_0  \\
   \end{array}
\right)
\end{equation}
where the upper part of the matrix consists of $\ell$ rows of coefficients of $F_1$, the lower part consists of $d$ rows of coefficients of $F_2$.
The resultant of $F_1$ and $F_2$ is the determinant of the Sylvester matrix of $F_1,F_2$, written as $\res_y(F_1,F_2)$. Specifically, if $\deg_y F_1=0$ or $\deg_y F_2=0$, we define
\begin{equation}\label{eq-3}
\res_y(F_1,F_2)=1.
\end{equation}

This is not the standard way to define the resultant of Equ. (\ref{eq-3}) as $1$. For the constant $c$, the standard is $\res_y(c,y^d)=c^d$. But we always make the leading coefficient non-zero after assigning a value or modulus $p$, and we only care if it is zero, so we simply define it as $1$. 
In this paper, we do not define the resultant for the cases $F_1=0$ or $F_2=0$.

The following are some facts. Denote $\LC_y(F_1)$ as the leading coefficient of $F_1$ w.r.t. $y$.

\begin{lemma}\cite[Lemma 4]{hu2021fast}\label{lm-2}
Let $\mathcal{R}$ be any integral domain and $F_1,F_2\in\mathcal{R}[x_1,\\ \dots,x_n,y]$ with $d=\deg_y(F_1)>0$ and $\ell=\deg_yF_2>0$. Let $a_d=\LC_y(F_1),b_{\ell}=\LC_y(F_2),R=\res_{y}(F_1,F_2),\overrightarrow{\alpha}\in \mathcal{R}^n$. Then
\begin{itemize}
\item[(\lowercase\expandafter{\romannumeral1})]   $a_d,b_{\ell},R$ are polynomials in $\mathcal{R}[x_1,\dots,x_n]$;
\end{itemize}
If $\mathcal{R}$ is a field and $a_d(\overrightarrow{\alpha})\neq 0$ and $b_{\ell}(\overrightarrow{\alpha})\neq 0$, then
\begin{itemize}
\item[(\lowercase\expandafter{\romannumeral2})]
    $\res_{y}(F_1(y,\overrightarrow{\alpha}),F_2(y, \overrightarrow{\alpha}))=R(\overrightarrow{\alpha})$ and  $\deg_{y}\gcd(F_1(y,\overrightarrow{\alpha}),F_2(y,\overrightarrow{\alpha}))>0\Longleftrightarrow \res_{y}(F_1(y,\overrightarrow{\alpha}),F_2(y,\overrightarrow{\alpha}))=0$;
\end{itemize}
If $\mathcal{R}=\Z$, $p$ is a prime and $\Phi_p(a_d)\neq 0$ and $\Phi_p(b_{\ell})\neq 0$ then
\begin{itemize}
\item[(\lowercase\expandafter{\romannumeral3})]$\Phi_p(F_1),\Phi_p(F_2),\Phi_p(R)$ are polynomials in $\F_p[x_1,\dots,x_n]$.
    \item[(\lowercase\expandafter{\romannumeral4})]$\res_y(\Phi_p(F_1),\Phi_p(F_2))=\Phi_p(R)$ and
\item[(\lowercase\expandafter{\romannumeral5})] $\deg_y\gcd(\Phi_p(F_1),\Phi_p(F_2))>0\Longleftrightarrow \res_y(\Phi_p(F_1),\Phi_p(F_2))=0$.
\end{itemize}
\end{lemma}

\begin{lemma}\label{the-14}
Let $A,B\in \mathcal{R}[x_1,\dots,x_n],\mathbf{s}=(s_1,\dots,s_n)\in\N^n$ and $R=\res_y(A_{(\mathbf{s},y)},\\ B_{(\mathbf{s},y)})$. Then $\deg R\leq 2\|\mathbf{s}\|_{\infty}\deg A\cdot\deg B$, where $\|\mathbf{s}\|_{\infty}=\max(s_1,\dots,s_n)$.
\end{lemma}

\begin{proof}
Assume $A_{(\mathbf{s},y)}=a_dy^d+\cdots+a_1y+a_0$ and $B_{(\mathbf{s},y)}=b_{\ell}y^{\ell}+\cdots+b_1y+b_0$, where $a_i,b_j\in \mathcal{R}[x_1,\dots,x_n]$. By the definition of $A_{(\mathbf{s},y)}$ and $B_{(\mathbf{s},y)}$, we have $d\leq \|\mathbf{s}\|_{\infty}\deg A$ and $\ell\leq \|\mathbf{s}\|_{\infty}\deg B$. As the Sylvester matrix is $d+\ell$ by $d+\ell$ matrix and $\deg a_i\leq \deg A,\deg b_j\leq \deg B$, the degree of $R$ is no more than $\ell\cdot \deg A+d\cdot \deg B$, which is $\leq 2\|\mathbf{s}\|_{\infty}\deg A\cdot\deg B$.
\end{proof}\qed

\subsection{Good point}
The main idea of our algorithm is mapping the entire problem to a simpler
domain via a homomorphism. Let
$$\sigma_\mathbf{p}:\Z[x_1,\dots,x_n,y]\rightarrow \Z[y]$$
be the ring homomorphism defined by evaluating $x_i=p_i,i=1,\dots,n$, where $\mathbf{p}=(p_1,\dots,p_n)$.
Let $F_1,F_2\in \Z[x_1,\dots,x_n,y]$ and $C=\gcd(F_1,F_2)$. Then $\sigma_\mathbf{p}(F_1)=F_1(\mathbf{p},y)$ and $\sigma_\mathbf{p}(F_2)=F_2(\mathbf{p},y)$. Suppose we compute the GCD of the univariate polynomials $\sigma_\mathbf{p}(F_1)$ and $\sigma_\mathbf{p}(F_2)$. Then, provided $\sigma_p(C) \ne 0$, it is
easy to see that
$$\sigma_\mathbf{p}(C)|\gcd(\sigma_\mathbf{p}(F_1),\sigma_\mathbf{p}(F_2)).$$

If $\deg\sigma_\mathbf{p}(C)=\deg\gcd(\sigma_\mathbf{p}(F_1),\sigma_\mathbf{p}(F_2))$,
$\gcd(\sigma_\mathbf{p}(F_1),\sigma_\mathbf{p}(F_2))$ retains parts of the information of $C$ to solve the problem in the original domain.
The leading coefficient $\LC_y(C)(\mathbf{p})$ of the GCD is not zero
   if $\sigma_\mathbf{p}(F_1)$ and $\sigma_\mathbf{p}(F_2)$ do not
decrease in degree, which leads to the following definition.

\begin{definition}\label{def-2}
Let $F_1,F_2\in \Z[x_1,\dots,x_n,y]$ and $C=\gcd(F_1,F_2)$. Let $\mathbf{p}=(p_1,\dots,p_n)\in\Z^n$. We say $\mathbf{p}$ is a {\em good point} for $F_1,F_2$ if $\sigma_{\mathbf{p}}\left( \LC_y(F_1) \right) \neq 0$, $\sigma_{\mathbf{p}}\left( \LC_y(F_2) \right) \neq 0$ and $\deg \sigma_\mathbf{p}(C) =\deg \gcd(\sigma_\mathbf{p}(F_1),\sigma_\mathbf{p}(F_2))$.
\end{definition}

Since $\sigma_\mathbf{p}(C)|\gcd(\sigma_\mathbf{p}(F_1),\sigma_\mathbf{p}(F_2))$,
if $\mathbf{p}$ is a good point for $F_1,F_2$, we have $\kappa\cdot\sigma_\mathbf{p}(C)=\gcd(\sigma_\mathbf{p}(F_1),\sigma_\mathbf{p}(F_2))$ for some constant $\kappa\in\mathbb{Z}$.

\begin{lemma}\label{the-13}
Let $F_1,F_2\in \Z[x_1,\dots,x_n,y]$ be non-zero and $\mathbf{p}=(p_1,\dots,p_n)\in \Z^n$. {\color{blue}Let} $$R=\res_y(F_1/\gcd(F_1,F_2),F_2/\gcd(F_1,F_2)).$$ Let $L=R\cdot \LC_y(F_1)\cdot \LC_y(F_2)$. Then
$\mathbf{p}$ is a good point for $F_1,F_2$ if and only if $L(\mathbf{p})\neq 0$.
\end{lemma}
\begin{proof}
``$\Longleftarrow$" Since $F_1/\gcd(F_1,F_2)$ and $F_2/\gcd(F_1,F_2)$ are coprime, $R\neq 0$ \cite[p288, Sylvester's Criterion]{Geddes92Algorithms}.
Assume $L(\mathbf{p})\neq 0$. Then $\LC_y(F_1)(\mathbf{p})\neq 0$, $\LC_y(F_2)(\mathbf{p})\neq 0$ and $R(\mathbf{p})\neq 0$. Assume $C=\gcd(F_1,F_2)$.
As $\LC_y(F_1/C)$ is a factor of $\LC_y(F_1)$, $\LC_y(F_1)(\mathbf{p})\neq 0$ implies that $\LC_y(F_1/C)(\mathbf{p})\neq 0$. For the same reason, $\LC_y(F_2/C)(\mathbf{p})\neq 0$. So by (\lowercase\expandafter{\romannumeral2}) of Lemma \ref{lm-2}, $R(\mathbf{p})=\res_y(F_1(\mathbf{p},y)/C(\mathbf{p},y),F_2(\mathbf{p},y)/C(\mathbf{p},y))\neq 0$.

Claim: $\deg_y \gcd(F_1(\mathbf{p},y)/C(\mathbf{p},y), F_2(\mathbf{p},y)/C(\mathbf{p},y))=0$.
Proof of claim: If $\deg_y(F_1/C)=0$ or $\deg_y(F_2/C)=0$, it is obvious. So assume $\deg_y(F_1/C)>0$ and $\deg_y(F_2/C)>0$.
According to (\lowercase\expandafter{\romannumeral2}) of Lemma \ref{lm-2}, $R(\mathbf{p})\neq 0$ implies that
$\deg_y\gcd(F_1(\mathbf{p},y)/C(\mathbf{p},y), F_2(\mathbf{p},y)/C(\mathbf{p},y))=0$. We proved it.
Therefore
 $F_1(\mathbf{p},y)/C(\mathbf{p},y)$ and $F_2(\mathbf{p},y)/C(\mathbf{p},y)$ have only a constant GCD. That is, $\gcd(F_1(\mathbf{p},y),F_2(\mathbf{p},y))$ is equivalent to $C(\mathbf{p},y)$ up to a non-zero constant,
which implies $\deg \sigma_\mathbf{p}(C) =\deg \gcd(\sigma_\mathbf{p}(F_1),\sigma_\mathbf{p}(F_2))$.
So $\mathbf{p}$ is a good point for $F_1,F_2$.

``$\Longrightarrow$" For the other direction, assume $\mathbf{p}$ is a good point for $F_1,F_2$, then $\LC(F_1)(\mathbf{p})\neq 0, \LC(F_2)(\mathbf{p})\neq 0$ and $\deg \sigma_\mathbf{p}(C) =\deg \gcd(\sigma_\mathbf{p}(F_1),\sigma_\mathbf{p}(F_2))$.
 Thus $\deg_y\gcd(F_1(\mathbf{p},y)/C(\mathbf{p},y), F_2(\mathbf{p},y)/C(\mathbf{p},y))=0$. By (\lowercase\expandafter{\romannumeral2}) of Lemma \ref{lm-2}, $R(\mathbf{p})=\res_y(F_1(\mathbf{p},y)/C(\mathbf{p},y),F_2(\mathbf{p},y)/C(\mathbf{p},y))\neq 0$. Therefore $L(\mathbf{p})=R(\mathbf{p})\cdot \LC(F_1)(\mathbf{p})\cdot \LC(F_1)(\mathbf{p})\neq 0$.
\end{proof}\qed

A single $\sigma_\mathbf{p}$ does not usually retain all the information
necessary to solve the problem in the original domain.
But in this paper, we will show that if $\mathbf{p}$ is chosen appropriately, then one image $\gcd(\sigma_\mathbf{p}(F_1),\sigma_\mathbf{p}(F_2))$ is enough to recover $C$.

\section{An algorithm for GCD of polynomials with given terms bound}
In this section, we give a new GCD algorithm for polynomials in $\Z[x_1,\dots,x_n]$ with a given term bound.
Let $A,B\in\Z[x_1,\dots,x_n]$ and $G=\gcd(A,B)$.
Then
$
G=\gcd(\MoCont(A),\MoCont(B))\cdot \gcd( \MoPrim(A),\MoPrim(B))
$ (cf. Def. \ref{def-3}),
where $\gcd(\MoCont(A),\MoCont(B))$ is the monomial content of $G$ and
$\gcd(\MoPrim(A),\MoPrim(B))$ is the monomial primitive part of $G$.
So to compute $G$, our algorithm is divided into three main parts.

\begin{description}
\item[Part 1] Compute the monomial content by computing $\gcd(\MoCont(A),\MoCont(B))$. This part is trivial because all polynomials in the computation are monomials.

\item[Part 2] Compute the monomial primitive part by computing $\MoPrim(G)=\gcd(\MoPrim(A),\MoPrim(B)).$ We explain the framework of this part in more detail in Section \ref{sec-a3} and summarize it as a subroutine algorithm (Algorithm \ref{alg-1} given in Section \ref{sec-32}).

\item[Part 3] Multiply the two parts to get the final result $G$. This part can be converted to the additions of exponents and multiplication of coefficients, as $\MoCont(G)$ is a term.
\end{description}

\subsection{Framework of GCD algorithm for monomial primitive polynomials}\label{sec-a3}
Assume $A,B\in \Z[x_1,\dots,x_n]$ are monomial primitive, so is their GCD.
To compute $G=\gcd(A,B)$, we first find a suitable vector $\mathbf{s}\in\N^n$, so that  the GCD of $F_1:=A_{(\mathbf{s},y)}$ and $F_2:=B_{(\mathbf{s},y)}$,  $G_{(\mathbf{s},y)}=\gcd(F_1,F_2)$ (cf. Lemma \ref{the-5}), is separated w.r.t $y$.
Second, we randomly choose different primes $\mathbf{p}=(p_1,\dots,p_n)$, so that $\deg G_{(\mathbf{s},y)}(\mathbf{p},y)=\deg \gcd(F_1(\mathbf{p},y),F_2(\mathbf{p},y)).$
By Theorem \ref{the-2}, if $p_1,\dots,p_n$ are distinct primes and $p_i$ does not divide any coefficient of $G_{(\mathbf{s},y)}$, then the content of $G_{(\mathbf{s},y)}(\mathbf{p},y)$ is $1$. So we have
$G_{(\mathbf{s},y)}(\mathbf{p},y)={\rm Prim}( \gcd(F_1(\mathbf{p},y),F_2(\mathbf{p},y))).$
Third, if $G=c_1M_1+\cdots+c_tM_t$, where $c_i\in\Z/\{0\}$ and $M_i,i=1,\dots,t$ are distinct monomials, then we have
$G_{(\mathbf{s},y)}(\mathbf{p},y)=c_1M_1(\mathbf{p})y^{k_1}+\cdots+c_tM_t(\mathbf{p})y^{k_t},$
where $k_i,i=1,\dots,t$ are different exponents and $c_iM_i(\mathbf{p})$ is the coefficient of $y^{k_i}$.
The exponents of $M_i$ can be obtained from the factorization of $c_iM_i(\mathbf{p})$ as $p_j,j=1,\dots,n$ are distinct primes and $p_j\nmid c_i$. If $c_iM_i(\mathbf{p})=c_ip_1^{e_{i,1}}\cdots p_n^{e_{i,n}}$, then $M_i=x_1^{e_{i,1}}\cdots x_n^{e_{i,n}}$.  Thus we obtain the exact form of $G=c_1x_1^{e_{1,1}}\cdots x_n^{e_{1,n}}+\cdots+c_tx_1^{e_{t,1}}\cdots x_n^{e_{t,n}}.$

We divide the algorithm into five main steps: 

\begin{description}
\item[Step a:] find a vector $\mathbf{s}\in\N^n$ so that $G_{(\mathbf{s},y)}$ is separated w.r.t $y$ using Theorem \ref{the-1}.

\item[Step b:] randomly choose different primes $p_1,\dots,p_n$ so that $\mathbf{p}=(p_1,\dots,p_n)$ is a $good\ point$ for $F_1,F_2$ (cf. Def.\ref{def-2})
and $p_i$ does not divide any integer coefficient of $G$.
\item[Step c:] compute $U:=\gcd(F_1(\mathbf{p},y),F_2(\mathbf{p},y))$  in   $\Z[y]$ using a univariate GCD algorithm.

\item[Step d:] compute the primitive part of $U$ to obtain $G_{(\mathbf{s},y)}(\mathbf{p},y)={\rm Prim}(U)$.
\item[Step e:] identify the terms of $G$ by dividing the coefficients of $G_{(\mathbf{s},y)}(\mathbf{p},y)$ by the primes $p_1, \dots, p_n$.
\end{description}

\subsection{Algorithm}\label{sec-32}
We give a GCD algorithm that works for monomial primitive polynomials. Algorithm \ref{alg-1} assumes a term bound $T$ for $G$ is given as input.

\begin{algorithm}[H]
\caption{Monomial primitive GCD of $\Z[x_1,\dots,x_n]$.}\label{alg-1}
\begin{algorithmic}[1]
\REQUIRE Monomial primitive polynomials $A,B\in \Z[x_1,\dots,x_n]$; a term bound $T\geq \#\gcd(A,B)$.
\ENSURE $G=\gcd(A,B)$ with probability $\geq \frac34$; or ``Failure".

\IF{$\#A=1$ or $\#B=1$}
\RETURN $1$.
\ENDIF

\STATE Randomly choose $\mathbf{s}=(s_1,\dots,s_n)\in\N^n$ from $[0,\frac92T(T-1)]^n$.\label{alg-1-1}

\STATE Compute $F_1:=A_{(\mathbf{s},y)}=A(x_1y^{s_1},\dots,x_ny^{s_n})/y^{k_1}$ and $F_2:=B_{(\mathbf{s},y)}=B(x_1y^{s_1},\dots,x_ny^{s_n})/y^{k_2}$ for some integers $k_1,k_2$ as defined in Eq. (\ref{eq-1}).\label{alg-1-2}

\STATE Find the first $N:=\max\{112\|\mathbf{s}\|_{\infty}D^2+112D+n-1,56n\lceil Tn(d\log_2 e+\log_2 C)\rceil+n-1\}$ primes and denote them as $\{p'_1,\dots,p'_N\}$.\label{alg-1-3}

\STATE Randomly choose $n$ different primes $\mathbf{p}:=(p_1,\dots,p_n)$ from $\{p'_1,\dots,p'_N\}$.\label{alg-1-4}

\IF{$\LC_y(F_1)(\mathbf{p})= 0$ or $\LC_y(F_2)(\mathbf{p})=0$)}\label{alg-1-31}
\RETURN ``Failure".
\ENDIF

\STATE Compute   $U:=\gcd(F_1(\mathbf{p},y),F_2(\mathbf{p},y))$ by univariate GCD algorithm (Algorithm \ref{alg-3}). \label{alg-1-5}

\IF{ $U=$``Failure"} \RETURN ``Failure";
\ELSE
\STATE assume
${\rm Prim}(U):=a_1y^{k_1}+\cdots+a_ty^{k_t}.$
\ENDIF

\FOR{$i=1$ to $t$}
\STATE factor the integer $a_i=c_ip_1^{e_{i,1}}\cdots p_n^{e_{i,n}}$ using only trial division by $p_1,\dots,p_n$.  \label{alg-1-6}
\ENDFOR

\STATE Let $G:=\sum_{i=1}^tc_ix_1^{e_{i,1}}\cdots x_n^{e_{i,n}}$.\label{alg-1-12}

\STATE Compute the exponents of $y$ in $G_{(\mathbf{s},y)}$.
\IF{$G_{(\mathbf{s},y)}$ has the same exponents in $y$ as $\Prim(U)$}
\RETURN $G$;\label{alg-1-23}
\ELSE \RETURN ``Failure". \label{alg-1-8}
\ENDIF
\end{algorithmic}
\end{algorithm}

\begin{theorem}
Algorithm \ref{alg-1} works correctly as specified.
\end{theorem}
\begin{proof}
As shown in Section \ref{sec-a3}, Algorithm \ref{alg-1} always returns the correct GCD if the following three conditions are met.

$\bullet$ The vector $\mathbf{s}=(s_1,\dots,s_n)$ chosen in Step \ref{alg-1-1} is a ``good" vector that separates $G(x_1y^{s_1},\dots,x_ny^{s_n})$ w.r.t. $y$.

$\bullet$ The point $\mathbf{p}=(p_1,\dots,p_n)$ chosen in Step \ref{alg-1-4} is  a $good\ point$ for $F_1$ and $F_2$, and no prime $p_i$
divides any integer coefficient of $G$.

$\bullet$ The univariate GCD $\gcd(F_1(\mathbf{p},y),F_2(\mathbf{p},y))$ is computed correctly by Algorithm \ref{alg-3} in Step \ref{alg-1-5}.

\noindent
We will prove the following three corresponding probabilities.

(1) $\mathbf{s}$ is a ``good" vector for  $G$ with probability $\geq \frac89$.

(2) $\mathbf{p}=(p_1,\dots,p_n)$ is a $good\ point$ for $F_1$ and $F_2$ and $p_i$ does not divide coefficients of $G$ with probability $\geq \frac{27}{28}$.

(3) The univariate GCD $\gcd(F_1(\mathbf{p},y),F_2(\mathbf{p},y))$ is computed correctly with probability $\geq \frac78$.

Thus our algorithm returns the correct GCD with probability $\geq \frac89\cdot \frac{27}{28}\cdot\frac{7}{8}=\frac34$.

According to Theorem \ref{the-1}, (1) is valid. According to Lemma \ref{lm-15}, (3) is valid.
Let's prove the probability of (2).

Let $R=\res_{y}(F_1/\gcd(F_1,F_2),F_2/\gcd(F_1,F_2))$
and $L=\LC_y(F_1)\cdot \LC_y(F_2)\cdot R.$ So $L \in \Z[x_1,\dots,x_n]$
and in Step \ref{alg-1-4}, if $\mathbf{p}$ satisfies $L(\mathbf{p})\neq 0$, then $\mathbf{p}$ is a $good\ point$ for $F_1,F_2$.
By Lemma \ref{the-14}, $\deg R\leq 2\|\mathbf{s}\|_{\infty}\deg A\deg B\leq 2\|\mathbf{s}\|_{\infty}D^{2}$, so we have $\deg L\leq\deg(\LC_y(F_1))+\deg(\LC_y(F_2))+\deg R\leq 2D+2\|\mathbf{s}\|_{\infty}D^2$.
As $\mathbf{p}$ is chosen from $\{p'_1,\dots,p'_N\}$, by Lemma \ref{lm-8},  $\mathbf{p}$ is not a $good\ point$ for $F_1$ and $F_2$ with probability $\leq \frac{\deg L}{N-n+1}\leq \frac{2(D+\|\mathbf{s}\|_{\infty}D^2)}{112(D+\|\mathbf{s}\|_{\infty}D^2)}\leq \frac{1}{56}.$
Let $K:=\lceil T(nd\log_2 e+\log_2 C)\rceil$. By Lemma \ref{lm-1}, the coefficient bound of $G$ is $e^{nd}C$. As there are at most $T$ terms in $G$,  there are at most $K$ different primes in all coefficients of $G$. So the probability that $p_i$ does not divide any coefficient of $G$ is $\geq \frac{N-K}{N}\cdot \frac{N-K-1}{N-1}\cdots \frac{N-K-n+1}{N-n+1}=(1-\frac{K}{N})(1-\frac{K}{N-1})\cdots (1-\frac{K}{N-n+1})$$
$$\geq (1-\frac{K}{N-n+1})^n\geq 1-\frac{nK}{N-n+1}.$
So the failure probability is $\leq \frac{nK}{N-n+1}\leq \frac{1}{56}$.
Therefore, the probability that the selected $\mathbf{p}=(p_1,\dots,p_n)$ satisfies (2) is $\geq 1-\frac{1}{56}-\frac{1}{56}=\frac{27}{28}.$
The correctness is proved. \qed
\end{proof}

Now we analyse the complexity of Algorithm \ref{alg-1}.

\begin{theorem}\label{the-7}
The complexity of Algorithm \ref{alg-1} is $O^\sim(n^2T^2d^2\log C+n^2dT_{\inp}\log T\\ \log C)$ bit operations, where $T_{\inp}:=\#A+\#B$.
\end{theorem}

\begin{proof}
In Step \ref{alg-1-2}, the complexity is $O^\sim(nT_{\inp}(\log T+\log D)+T_{\inp}\log C)$ bit operations.
In Step \ref{alg-1-3}, finding the first $N$ primes costs $O^\sim(N)$ bit operations \cite[P.500, Them. 18.10]{GaGer13}, which is $O^\sim(T^2D^2+n^2Td+n^2T\log C)$ bit operations.
In Step \ref{alg-1-5}, polynomials $F_1(\mathbf{p},y)$ and $F_2(\mathbf{p},y)$ are computed which costs $O^\sim(nT_{\inp}\log d\log (CP^D))$ bit operations where $P:=\max\{p_1,\dots,p_n\}$. The complexity of computing the univariate GCD $U(y)$ in Step \ref{alg-1-5} is $O^\sim(nSdD+SD^2\log P+SD\log C)$ bit operations by Lemma \ref{lm-13}, where $S:=\max\{s_1,\dots,s_n\}$. As $P$ is $O^\sim(T^2D^2+n^2Td+n^2T\log C)$ and $S$ is $O(T^2)$, the bit complexity is $O^\sim(n^2T^2d^2\log C)$.
In Step \ref{alg-1-6}, we factor $a_i$ by trivial division which costs $O^\sim(nTd\log (e^{nd}CP^D))$ bit operations, that is,  $O^\sim(n^2Td^2\log C)$ bit operations.
\end{proof}\qed

\vspace{11pt}
Our algorithm is randomized of Monte Carlo type. If $H$ is the output of algorithm \ref{alg-1}, then $H$ may not be the correct $\gcd(A,B)$. But if we verify 1: $H|A$ and $H|B$; 2: $\deg_{x_i}H=\deg_{x_i}\gcd(A,B),i=1,\dots,n$, then $H$ is always the correct GCD.
We state this as the following theorem.

\begin{theorem}\label{the-15}
Assume $A,B\in\Z[x_1,\dots,x_n]$ are monomial primitive polynomials as the inputs of Algorithm \ref{alg-1}, and $H$ is the output of Algorithm \ref{alg-1}. If $H|A$ and $H|B$ and
$\deg_{x_i}H=\deg_{x_i}\gcd(A,B),i=1,2,\dots,n$, then $H$ is a GCD of $A,B$.
\end{theorem}
\begin{proof}
$H|A$ and $H|B$ imply $H|\gcd(A,B)=G$. As $\deg_{x_i}H=\deg_{x_i}\gcd(A,B)$ for all $i=1,\dots,n$, there exists an integer $\kappa$ that makes $G=\kappa H$. Because $A,B$ are both primitive, so is $G$, therefore $\kappa = \pm 1$.  We proved it. \qed
\end{proof}

The  degree condition in Theorem \ref{the-15} is necessary because our new substitution $x_i \rightarrow p_i y^{s_i}$, $i=1,2,\dots,n$ may cause a factor of the gcd $G$ to vanish.  We give an example to illustrate this.

\begin{example}
    Let $A=B=(5x_1 - 3x_2+1)(2x_1x_2+7)$.  If Algorithm \ref{alg-1} chooses $s = [1,1]$ in Step \ref{alg-1-1} and ${\bf p}=[3,5]$ in Step \ref{alg-1-4} then the first factor $5x_1-3x_2+1$ is mapped to $1$ and the second factor $2x_1x_2+7$ is mapped to $30y^2+7$ which is separated w.r.t. $y$ and Algorithm \ref{alg-1} will output $G=2x_1x_2+7$ in Step \ref{alg-1-23}.
\end{example}

\subsection{Complexity analysis of univariate polynomials GCD}

We will explain Part $\mathbf{c}$ in more detail.
To compute $\gcd(F_1(\mathbf{p},y),F_2(\mathbf{p},y))$,  any univariate GCD algorithm can be used, such as 1: modular GCD in $\Z[x]$: big prime version \cite[p166, Alg. 6.34]{GaGer13};  2: modular GCD in $\Z[x]$: small primes version \cite[p170, Alg. 6.38]{GaGer13}; 3: Hensel lifting.
In our algorithm, due to the large coefficients of $F_1(\mathbf{p},y),F_2(\mathbf{p},y)$, we choose the Hensel lifting method to do this. This is because finding a big prime has a large complexity, whereas the Hensel lifting can start with a small prime.
The small primes version of the modular GCD algorithm has the same advantage, and in fact, it computes faster in practice. We did not choose it because it requires multiple primes and the analysis of the failure probability is more complex.

We have the following lemmas.
\begin{lemma}\label{lm-10}
Let $F,G,H\in\F_p[x]$, where $p$ is a prime and $\F_p$ is a finite field. Assume the GCD of $F,G,H$ is $1$. Then there are at most $\deg(F)$ elements $\alpha\in \F_p$, making $\gcd(F,G+\alpha H)\neq 1$.
\end{lemma}
\begin{proof}
Assume $\mathcal{E}$ is the splitting field of $F$ over $\F_p$. Let $\beta_1,\dots,\beta_k$ be all the different roots of $F$ in $\mathcal{E}$, where $k\leq \deg(F)$. Assume $r=\gcd(F,G+\alpha H)\neq 1$.
As $r$ is a factor of $F$, there exists some $\beta_i$ such that $r(\beta_i)=0$.
Thus $G(\beta_i)+\alpha H(\beta_i)=0$.

We claim that $H(\beta_i)\neq 0$.
Towards a contradiction assume $H(\beta_i)=0$, then $G(\beta_i)\neq 0$ as $F,G,H$ have no nontrivial common factor. Thus $G(\beta_i)+\alpha H(\beta_i)\neq 0$, which is a contradiction.
So $\alpha=-\frac{G(\beta_i)}{H(\beta_i)}$. As there are at most $k$ different $\beta_i$ and $k\leq \deg (F)$, we proved it.
\end{proof}\qed

\vspace{5pt}
This lemma is used to construct two coprime polynomials from three polynomials with no non-trivial common factor.  This method will be used in Step \ref{alg-3-131} of Algorithm \ref{alg-3} and was first introduced by Wang \cite{Wang80b}.

\begin{lemma}\label{lm-1}\cite{gelfond2015} Let $F$ be a polynomial in $\Z[x_1,\dots,x_n]$ and let $d_i$ be the degree of $F$ in $x_i$. If $G$ is any factor of $F$
over $\Z$ then $\|G\|_\infty\leq e^{d_1+\cdots+d_n}\|F\|_\infty$ where $e$ is the base of natural logarithm.
\end{lemma}

\begin{lemma}\label{lm-11}
Assume $A,B,G\in \Z[x_1,\dots,x_n]$ and $G=\gcd(A,B)$. Let $p_1,\dots,p_n\in\N$ be integers and $\mathbf{s}=(s_1,\dots,s_n)\in\N^n$. Assume $P:=\max\{p_1,\dots,p_n\}$ and $S:=\max\{s_1,\dots,s_n\}$. If $\max\{\|A\|_{\infty},\|B\|_{\infty}\}\leq C$, then $\|G(p_1y^{s_1},\dots,p_ny^{s_n})\|_{\infty}\leq e^{nd}(d+1)^nP^{D}C$,
where $D=\max\{\deg A,\deg B\}$ and $d$ is the partial degree bound of $A$ and $B$.
\end{lemma}
\begin{proof}
Assume $G=c_1M_1+\cdots+c_tM_t.$
Then $\|G(p_1y^{s_1},\dots,p_ny^{s_n})\|_{\infty}\leq |c_1|M_1(\mathbf{p})+\cdots+|c_t|M_t(\mathbf{p})\leq tP^{D}\|G\|_{\infty}$.  By Lemma \ref{lm-1}, we have $\|G\|_\infty\leq e^{nd}C$. As $t\leq (d+1)^n$, $\|G(p_1y^{s_1},\dots,p_ny^{s_n})\|_{\infty}\leq e^{nd}(d+1)^nP^{D}C$.
\end{proof}\qed

\begin{lemma}\cite[In the proof of Th. 18.8]{GaGer13}\label{lm-3}
The total number of primes in the range $(N,\dots,2N]$ is at least $\frac12N/\ln N$ when $N\geq 6$.
\end{lemma}

Let $Q$ be a $k\times k$ matrix with $Q_{i,j}\in\Z$. Hadamard's bound $\H(Q)$ for $|\det(Q)|$ is
$|\det Q| \leq \prod_{i=1}^k \sqrt{\sum_{j=1}^kQ^2_{i,j}}=:\H(Q)$.
For a polynomial $f\in\Z[x_1,\dots,x_n]$, we define $\|f\|_1=$ sum of the absolute values of the coefficients.

\begin{lemma}\label{lm-7}\cite{lossers1974hadamard}
Let $U$ be a $k\times k$ matrix with entries $U_{i,j}\in\Z[y]$. Let $Q$ be the $k\times k$ integer matrix with $Q_{i,j}=\|U_{i,j}\|_1$. Then $\|\det U\|_\infty\leq \H(Q)$.
\end{lemma}

The result can be easily generalized to the multivariate case, as we can always use the Kronecker substitution $x_1=x,x_2=x^m,\dots,x_n=x^{m^{n-1}}$. Choose $m$ large enough, then each polynomial in the entries and $\det A$ are one-to-one corresponding to univariate polynomials.

\begin{lemma}\label{lm-5}
Let $f,g\in\Z[x_1,\dots,x_n]$ be primitive and coprime.  Let $d$ be the partial degree bound of $f,g$ and $\|f\|_{\infty},\|g\|_{\infty}\leq C$.
Let $R=\res_{x_n}(f,g)$. Then $R\neq 0$ and $\|R\|_{\infty}\leq (d+1)^{2nd-d}C^{2d}.$
\end{lemma}
\begin{proof}
If $\deg_{x_n}f=0$ or $\deg_{x_n}g=0$, then $R=1$ (see Eq. (\ref{eq-3})). We proved it. So now assume $\deg_{x_n}f>0$ and $\deg_{x_n}g>0$.
As $f$ and $g$ are coprime, $R\neq 0$ \cite[p288, Sylvester's Criterion]{Geddes92Algorithms}. The partial degree of $R$ is $\leq 2d^2$ by the definition of resultant. We analyse the coefficient bound of $R$. Assume $U_{i,j}$ are the entries of ${\rm Syl}(f,g)$, the $Sylvester\ matrix$ of $f,g$ w.r.t $x_n$, and $Q_{i,j}=\|U_{i,j}\|_1$. Regard $f,g\in \Z[x_1,\dots,x_{n-1}][x_n]$. As there are at most $(d+1)^{n-1}$ terms in each coefficient of $f,g$ w.r.t $x_n$, $Q_{i,j}\leq (d+1)^{n-1}C$.  By Lemma \ref{lm-7}, $\|R\|_{\infty}=\|\det{\rm Syl}(f,g)\|_{\infty}\leq \H(Q)\leq \prod_{i=1}^{2d}\sqrt{\sum_{i=1}^{d+1}(d+1)^{2(n-1)}C^2}=(d+1)^{2nd-d}C^{2d}$.
\end{proof}\qed

In the following, denote $H_1:=A_{(\mathbf{s},y)}(\mathbf{p},y)=F_1(\mathbf{p},y)$, $H_2:=B_{(\mathbf{s},y)}(\mathbf{p},y)=F_2(\mathbf{p},y)$. Assume $H_1\neq 0$ and $H_2\neq 0$. 

\begin{lemma}\label{lm-9}
Let $A,B\in \Z[x_1,\dots,x_n]$, $\mathbf{p} =(p_1,\dots,p_n)\in\N^n$ and $P=\max\{p_1,\dots,p_n\}$. Let $\mathbf{s}=(s_1,\dots,s_n)\in\N^n$ and $S=\max\{s_1,\dots,s_n\}$. Let $D= \max(\deg A,\deg B)$ and $d=\max_{i=1}^n \max(\deg_{x_i}A,\deg_{x_i}B)$. Let $C=\max \{\|A\|_{\infty},\|B\|_{\infty}\}$. Let $R=\res_{y}(H_1/\gcd(H_1,H_2),H_2/\gcd(H_1,H_2))$.  Let $N\geq 16(SD\ln(SD+1)+SD(2SD+2)+n(2SD+2)\ln(d+1)+D(2SD+2)\ln P+(2SD+2)\ln C)$. If we randomly choose a prime $p$ in $(N,2N]$, then $p\nmid R\cdot \LC(H_1)\cdot \LC(H_2)$ with probability $\geq \frac{7}{8}$.
\end{lemma}

\begin{proof}
As $\|A\|_{\infty}\leq C$ and $\#A\leq (d+1)^n$, $\|H_1\|_{\infty}\leq (d+1)^nP^DC$. Similarly, $\|H_2\|_{\infty}\leq (d+1)^nP^DC$. Denote $\widetilde{H}_1:=H_1/\gcd(H_1,H_2)$ and $\widetilde{H}_2:=H_2/\gcd(H_1,H_2)$. As $\deg H_1,\deg H_2\leq SD$, by Lemma \ref{lm-1}, $\|\widetilde{H}_1\|_{\infty},\|\widetilde{H}_2\|_{\infty}\leq  e^{SD}(d+1)^nP^DC$. As $\widetilde{H}_1$ and $\widetilde{H}_2$ are coprime, by Lemma \ref{lm-5}, $|R|=\|R\|_{\infty}\leq (SD+1)^{SD}[e^{SD}(d+1)^nP^DC]^{2SD}$.
Denote $\kappa:=|R|\cdot |\LC(H_1)|\cdot |\LC(H_2)|$.
Then $\kappa\leq (SD+1)^{SD}[e^{SD}(d+1)^nP^DC]^{2SD+2}$.
There are at most $\frac{\ln \kappa}{\ln N}$ different primes in $\kappa$ greater than $N$. By Lemma \ref{lm-3}, there are at least $\frac12N/\ln N$ primes in $(N,2N]$. So if randomly choose a prime $p$ in $(N,2N]$, the probability that $p\nmid \kappa$ is $\geq 1-\frac{\ln \kappa/\ln N}{\frac{1}{2}N/\ln N}=1-\frac{2\ln \kappa}{ N}$. As $N\geq 16(SD\ln(SD+1)+SD(2SD+2)+n(2SD+2)\ln(d+1)+D(2SD+2)\ln P+(2SD+2)\ln C)$, the probability $\geq 1-\frac{1}{8}=\frac{7}{8}$.
We proved it.
\end{proof}\qed

\subsubsection{Algorithms}

The Hensel lifting algorithm is called by our GCD Algorithm \ref{alg-3}.

\begin{algorithm}[H]
\caption{Hensel lifting \cite[Alg. 15.17]{GaGer13}}\label{alg-2}
\begin{algorithmic}
\REQUIRE A prime $p\in \N$, $f\in \Z[x]$ of degree $d\geq 1$ such that $\LC(f)$ is a unit modulo $p$;
monic nonconstant polynomials $f_1,f_2\in \Z[x]$ that are coprime modulo $p$ and satisfy $f\equiv \LC(f)f_1 f_2 \mod p$; an integer $\ell \in \N$.

\ENSURE Monic polynomials $f^*_1,f^*_2\in \Z[x]$ with  $f\equiv \LC(f)f^*_1 f^*_2 \mod p^{\ell}$ and $f^*_i\equiv f_i \mod p$ for $i=1,2$.
\end{algorithmic}
\end{algorithm}

\begin{lemma}\cite[The.15.18]{GaGer13}\label{lm-12}
Algorithm \ref{alg-2} works correctly as specified.
If $p\in \N$ is prime, $\|f \|_{\infty} < p^\ell$, and $\|f_i\|_{\infty} < p$ for all $i$, then the algorithm takes $O^\sim(d\ell \log p)$ bit operations.
\end{lemma}

Let's show the details of how to compute $\gcd(F_1(\mathbf{p},y),F_2(\mathbf{p},y))$  in the Step \ref{alg-1-5} of Algorithm \ref{alg-1}. Assume $\mathbf{p}$ is a $good\ point$ for $F_1,F_2$, which makes $\deg(\gcd(F_1(\mathbf{p},y),F_2(\mathbf{p},y)))=\deg (G_{(\mathbf{s},y)}(\mathbf{p},y))$. So the primitive part of $G_{(\mathbf{s},y)}(\mathbf{p},y)$  is the same as the primitive part of $\gcd(F_1(\mathbf{p},y),F_2(\mathbf{p},y))$.

The following are the main steps.
\begin{itemize}
\item find a ``good" prime $p$ and  reduce $F_1(\mathbf{p},y)$ and $F_2(\mathbf{p},y)$ to $A_p:=F_1(\mathbf{p},y) \mod p$ and $B_p:=F_2(\mathbf{p},y) \mod p$.
\item compute the GCD $U_p=\gcd(A_p,B_p)$ in $\F_p[y]$.

\item  lift $U_p$ to $G_{(\mathbf{s},y)}(\mathbf{p},y)\Delta$ by Hensel lifting method, where $\Delta\in\Z$ and the leading coefficient of  $G_{(\mathbf{s},y)}(\mathbf{p},y)\Delta$ is the same as the leading coefficient of  $F_1(\mathbf{p},y)$.
\item  compute the primitive part of $G_{(\mathbf{s},y)}(\mathbf{p},y)\Delta$ to obtain the primitive part of $G_{(\mathbf{s},y)}(\mathbf{p},y)$.
\end{itemize}

\begin{algorithm}[H]
\caption{Univariate polynomial GCD.}\label{alg-3}
\begin{algorithmic}[1]
\REQUIRE An integer vector $\mathbf{s}=(s_1,\dots,s_n)\in \N^n$; an integer vector $\mathbf{p}=(p_1,\dots,p_n)\in \N^n$ and two polynomials $F_1(\mathbf{p},y),F_2(\mathbf{p},y)\in \Z[y]$ so that $\deg G_{(\mathbf{s},y)}(\mathbf{p},y)=\deg \gcd(F_1(\mathbf{p},y),F_2(\mathbf{p},y))$;
a bound $C\geq \max(\|A\|_{\infty},\|B\|_{\infty})$.
\ENSURE With probability $\geq 7/8$, return the primitive part of $G_{(\mathbf{s},y)}(\mathbf{p},y)$.

\STATE Find the parameters $D,d,P,S$ as defined in Lemma \ref{lm-9}.

\STATE $N:=\lceil 16(SD\ln(SD+1)+SD(2SD+2)+n(2SD+2)\ln(d+1)+D(2SD+2)\ln P+(2SD+2)\ln C)\rceil$.\label{alg-3-2}

\REPEAT
\STATE randomly choose a prime $p$ from $(N,2N]$ \label{alg-3-3} \UNTIL {$p \not | \, LC(F_1(\mathbf{p},y))$ or $p \not | \, \LC(F_2(\mathbf{p},y))$.}

\STATE Compute $A_p:=F_1(\mathbf{p},y)\mod p$ and $B_p:=F_2(\mathbf{p},y)\mod p$.\label{alg-3-4}

\STATE  Compute the GCD $U_p:=\gcd(A_p,B_p)$ in $\F_p[y]$ and $\widetilde{A} = A_p \div U_p$ and $\widetilde{B}=B_p \div U_p$. \label{alg-3-5}

 \IF{$\gcd(U_p,\widetilde{A})=1$\label{alg-3-6}}
     \STATE let $f_1:=U_p,f_2:=\widetilde{A},f:=F_1(\mathbf{p},y)$;
 \ELSIF {$\gcd(U_p,\widetilde{B})=1$\label{alg-3-9}}
  \STATE  let $f_1:=U_p,f_2:=\widetilde{B},f:=F_2(\mathbf{p},y)$;
  \ELSE
    \REPEAT
    \STATE randomly choose $\alpha\in \F^{*}_p$ \label{alg-3-132}\UNTIL {$\gcd(U_p,\widetilde{A}+\alpha \widetilde{B})=1$.}
    \STATE let $f_1:=U_p, f_2:=\widetilde{A}+\alpha \widetilde{B},f:=F_1(\mathbf{p},y)+\alpha F_2(\mathbf{p},y)$;\label{alg-3-131}
 \ENDIF \label{alg-3-13}

\STATE Let $\ell:=\lceil\log _p (2e^{nd}(d+1)^{2n}P^{2D}C^2)\rceil+1$.\label{alg-3-15}

\STATE Call Algorithm \ref{alg-2} with input $f,f_1,f_2,p$ and $\ell$ such that $f\equiv \LC(f) f_1f_2 \mod p$. Let $f_1^*,f_2^*$ be the output such that $f\equiv \LC(f)f_1^*f_2^* \mod p^{\ell}$ and $f_1^*\equiv f_1 \mod p,f_2^*\equiv f_2 \mod p$.\label{alg-3-16}

\STATE Compute the polynomial $f_{\rm sym}$ such that coefficients of $f_{\rm sym}\equiv\LC(F_1(\mathbf{p},y))f^*_1 \mod p^{\ell}$ and are
in the range $[-p^{\ell}/2,p^{\ell}/2]$.
\RETURN the primitive part of $f_{\rm sym}$\label{alg-3-17}.

\end{algorithmic}
\end{algorithm}

\begin{lemma}\label{lm-15}
Algorithm \ref{alg-3} works correctly as specified.
\end{lemma}
\begin{proof}
To lift $U_p\in \F_p[y]$ to the  polynomial $\gcd(H_1,H_2)\in \Z[y]$, we should ensure that $\deg(U_p)=\deg (\gcd(H_1,H_2))$, where $H_i=F_i(\mathbf{p},y)$.
By Lemma \ref{lm-2}, $p$ should satisfy the following two conditions.
\begin{itemize}
\item $p\nmid \LC(H_1)$ and $p\nmid \LC(H_2)$;
\item $p\nmid \res_y(H_1/\gcd(H_1,H_2),H_2/\gcd(H_1,H_2))=:R$.
\end{itemize}

As the choice of $N$ in Step \ref{alg-3-2}, by Lemma \ref{lm-9},  $p\nmid R\cdot \LC(H_1)\cdot \LC(H_2)$ with probability $\geq \frac{7}{8}$.
Thus by (\lowercase\expandafter{\romannumeral5}) of Lemma \ref{lm-2}, $U_p$ computed in Step \ref{alg-3-5} has $\deg(U_p)=\deg (\gcd(H_1,H_2))$ with probability $\geq \frac78$.

In Step \ref{alg-3-6}-\ref{alg-3-13}, we construct three polynomials $f,f_1,f_2$ which satisfy $f\equiv\LC(f)\cdot f_1f_2 \mod p$, and where $f_1(=U_p)$ and $f_2$ are coprime.

The correctness of construction in Step \ref{alg-3-131} comes from Lemma \ref{lm-10}, since neither $U_p,\widetilde{A}$ nor $U_p,\widetilde{B}$ are coprime, the GCD of $U_p,\widetilde{A},\widetilde{B}$ is $1$.

Then in Step \ref{alg-3-16}, we lift $f_1,f_2$ to $f^*_1,f^*_2$ so that $f\equiv \LC(f)f^*_1f^*_2 \mod p^{\ell}$. As $G_{(\mathbf{s},y)}(\mathbf{p},y)$ divides $H_1$, $\LC(G_{(\mathbf{s},y)}(\mathbf{p},y))$ divides $\LC(H_1)$. 

Assume $\Delta\cdot \LC(G_{(\mathbf{s},y)}(\mathbf{p},y))=\LC(H_1)$ for some integer $\Delta$.
As $f^*_1$ is corresponding to $U_p$ and monic, then $\LC(H_1)\cdot f^*_1\equiv \Delta\cdot G_{(\mathbf{s},y)}(\mathbf{p},y)\mod p^{\ell}$.

The size of $\Delta$ is $\|H_1\|_{\infty}$, which is  $\le$ $(d+1)^nP^DC$. The size of the coefficients of $G_{(\mathbf{s},y)}(\mathbf{p},y)$ are $\le$ $e^{nd}(d+1)^nP^{D}C$ by Lemma \ref{lm-11}. Thus the size of the coefficients of $\Delta\cdot G_{(\mathbf{s},y)}(\mathbf{p},y)$ are at most $e^{nd}(d+1)^{2n}P^{2D}C^2$.

In Step \ref{alg-3-15}, as $\ell=\lceil\log _p (2e^{nd}(d+1)^{2n}P^{2D}C^2)\rceil+1$,
$p^{\ell}> 2\|\Delta\cdot G_{(\mathbf{s},y)}(\mathbf{p},y)\|_{\infty}$.
Therefore, if we let the max-norm of $\LC(H_1)\cdot f^*_1$ not greater than $p^{\ell}/2$, then $\Delta\cdot G_{(\mathbf{s},y)}(\mathbf{p},y)$ is equal to $\LC(H_1)\cdot f^*_1$.

In Step \ref{alg-3-17}, we compute the primitive part of $\Delta\cdot G_{(\mathbf{s},y)}(\mathbf{p},y)$ by divided by the GCD of its coefficients. Then we obtain the primitive part of $G_{(\mathbf{s},y)}(\mathbf{p},y)$.
\end{proof}\qed

\begin{lemma}\label{lm-13}
 The complexity of Algorithm \ref{alg-3} is $O^\sim(nSdD+SD^2\log P+SD\log C)$ bit operations.
\end{lemma}
\begin{proof}
In Step \ref{alg-3-3}, finding  a prime in $[N+1,2N]$ costs $O^\sim(\log^3 N)$ bit operations \cite[The.18.8]{GaGer13}. As the height of $N$ is $O(\log S+\log D+\log n+\log \log P+\log \log C)$, the complexity is $O^\sim((\log S+\log D+\log n+\log \log P+\log \log C)^3)$ bit operations.

In Step \ref{alg-3-4}, we compute $A_p$ and $B_p$. As $\|F_1(\mathbf{p},y)\|_{\infty}$ and $\|F_2(\mathbf{p},y)\|_{\infty}$  are both $\le$ $(d+1)^nP^DC$, and $A_p,B_p$ have at most $SD+1$ terms, the complexity is $O^\sim(nSD\log d+SD^2\log P+SD\log C)$ bit operations.

In Step \ref{alg-3-5}, \ref{alg-3-6} and \ref{alg-3-9}, we compute $\gcd(A_p,B_p)$,  $A_p \div U_p$, $B_p \div U_p$, $\gcd(U_p,\widetilde{A})$ and $\gcd(U_p,\widetilde{B})$ in $\F_p[y]$. As the degree of $A_p,B_p,U_p,\widetilde{A},\widetilde{B}$ are $O(SD)$, the complexity is $O^\sim(SD\log p)$ bit operations~\cite[Cor.3.14]{P2010Algebraic}.
As $p$ is chosen from $[N+1,2N]$ and $\log N$ is $O(\log n+\log S+\log D+\log\log P+\log \log C)$, $O^\sim(SD\log p)$ is  $O^\sim(SD\log n+SD\log\log P+SD\log\log C)$ bit operations.
In Step \ref{alg-3-132}, by Lemma \ref{lm-10}, if we randomly choose $\alpha\in[0,p-1]$, then the probability of $\gcd(U_p,\widetilde{A}+\alpha \widetilde{B})=1$ is $\geq 1-\frac{\deg U_p}{p}$. As $p\geq 16SD$ and $\deg U_p\leq SD$, the probability is $\geq \frac{15}{16}$. As the bit complexity of computing $\gcd(U_p,\widetilde{A}+\alpha \widetilde{B})$ is $O^\sim(SD\log p)$, the expected complexity is also $O^\sim(SD\log p)$ bit operations, that is, $O^\sim(SD\log n+SD\log\log P+SD\log\log C)$ bit operations.
In Step \ref{alg-3-15}, as $\ell=\lceil\log_p (2e^{nd}(d+1)^{2n}P^{2D}C^2)\rceil$+1, $\ell \log p$ is in $O(\log (2e^{nd}(d+1)^{2n}P^{2D}C^2))$. The degrees of $f,f_1,f_2$ are $O(SD)$, by Lemma \ref{lm-12}, the complexity of Step \ref{alg-3-16} is $O^\sim(SD\ell \log p)$ bit operations, which is $O^\sim(SD\log (2e^{nd}(d+1)^{2n}P^{2D}C^2))=O^\sim(nSdD+SD^2\log P+SD\log C)$ bit operations.
In Step \ref{alg-3-17}, we compute at most $O(SD)$ integer GCDs. As the coefficients of $\LC(F_1(\mathbf{p},y))f_1$ is $O(p^{\ell})$, the complexity is $O^\sim(SD\ell \log p)$ bit operations which is the same as that in Step \ref{alg-3-16}.
\end{proof}\qed

\section{Verifying whether $H\overset{\text{?}}=\gcd(A,B)$}
Algorithm \ref{alg-1} is a randomized algorithm that outputs a candidate polynomial $H$. In this section, based on Theorem \ref{the-15}, we give a randomized algorithm to test whether $H\overset{\text{?}}=\gcd(A,B)$.

From Theorem \ref{the-15}, if $A,B$ are primitive and $\LC(H)>0$ then $H=\gcd(A,B)$ if and only if \begin{enumerate}
\item $\deg_{x_i}H=\deg_{x_i}\gcd(A,B),i=1,\dots,n$ and
\item $H|A$ and $H|B$.
\end{enumerate}

The framework of this section is as follows. In Section \ref{sec-4-2}, we illustrate how to compute the partial degrees of $\gcd(A,B)$.
In Section \ref{section-4-1}, we illustrate how to test $H|A$ and $H|B$, and then provide an algorithm to test if $H\overset{\text{?}}=\gcd(A,B)$.

Assume $F\in \Z[x_1,\dots,x_n]$ and $\mathbf{a}=(a_1,\dots,a_{n})\in\Z^n$. For $i=1,\dots,n$, denote
$$(\Breve{\mathbf{a}},x_i):=(a_1,\cdots,a_{i-1}, x_{i},a_{i+1},\dots,a_{n})$$
$$F(\Breve{\mathbf{a}},x_i):=F(a_1,\cdots,a_{i-1}, x_{i},a_{i+1},\dots,a_{n})$$

We give the following lemma, which will be used for our proofs below.
\begin{lemma}\label{lm-16}
Assume $A,B,G\in\Z[x_1,\dots,x_n]$ and $G=\gcd(A,B)$. Let $d_i=\deg_{x_i}G$. Then for each $i=1,\dots,n$,  there exists a polynomial $$\Gamma_i\in \Z[x_1,\dots,x_{i-1},x_{i+1},\dots,x_{n}]$$ so that for any prime $p\in\N$ and $\mathbf{a}=(a_1,\dots,a_{n})\in \mathbb{F}^{n}_p$, if $\Gamma_i(\mathbf{a})\mod p\neq 0$, then
$$\Phi_p(G(\Breve{\mathbf{a}},x_i)) \sim\gcd(\Phi_p(A(\Breve{\mathbf{a}},x_i)),\Phi_p(B(\Breve{\mathbf{a}},x_i)))$$ and $d_i=\deg_{x_i}\gcd(\Phi_p(A(\Breve{\mathbf{a}},x_i)),\Phi_p(B(\Breve{\mathbf{a}},x_i))).$
Furthermore, let $d$ be the partial degree bound of $A,B$ and $\|A\|_{\infty},\|B\|_{\infty}\leq C$. Then
(1) $\deg_{x_j}\Gamma_i\leq 2d^2+2d,j=1,2,\dots,n$ and (2) the leading coefficient of $\Gamma_i$ satisfies $|\LC(\Gamma_i)|\leq (d+1)^{(2n-1)d}e^{2nd^2}C^{2d+2}$.

\end{lemma}
\begin{proof}
Let $R=\res_{x_i}(A/G,B/G)$ and $$\Gamma_i:=R\cdot \LC_{x_i}(A)\cdot \LC_{x_i}(B)\in \Z[x_1,\dots,x_{i-1},x_{i+1},\dots,x_n].$$ By Lemma \ref{lm-5}, $\Gamma_i\neq 0$.
If $\Gamma_i \mod p\neq 0$, then by (\lowercase\expandafter{\romannumeral5}) of Lemma \ref{lm-2}, $\deg_{x_i}\gcd(\Phi_p(A)/\Phi_p(G),\Phi_p(B)/\Phi_p(G))=0$, which means $\gcd(\Phi_p(A),\Phi_p(B))=\Phi_p(G)\cdot \kappa$ for some $\kappa\in\F_p[x_1,\dots,x_{i-1},x_{i+1},\dots,x_n]$.
Let $\mathbf{a}=(a_1,\dots,a_{n})\in\F_p^{n}$.
If $\Gamma_i(\mathbf{a}) \mod p\neq 0$,
then by Lemma \ref{the-13}, $(a_1,\dots,a_{i-1},a_{i+1},\dots,a_n)$ is a good point for $\Phi_p(A)$ and $\Phi_p(B)$, therefore, $$\deg_{x_i}\Phi_p(G(\Breve{\mathbf{a}},x_i))=\deg_{x_i}\gcd(\Phi_p(A(\Breve{\mathbf{a}},x_i),\Phi_p(B(\Breve{\mathbf{a}},x_i))$$ as $\deg_{x_i}\kappa=0$.
So $\Phi_p(G(\Breve{\mathbf{a}},x_i))\sim\gcd(\Phi_p(A(\Breve{\mathbf{a}},x_i)),\Phi_p(B(\Breve{\mathbf{a}},x_i))$.  As $\LC_{x_i}(G)(\mathbf{a})\mod p\neq 0$, $d_i=\deg_{x_i}\Phi_p(G(\Breve{\mathbf{a}},x_i))=\deg_{x_i}\gcd(\Phi_p(A(\Breve{\mathbf{a}},x_i)),\Phi_p(B(\Breve{\mathbf{a}},x_i)))$. We proved it.

As $d$ is the partial degree bound of $A,B$,  by the definition of resultant, the partial degree bound of $R$ is $\leq 2d^2$. The partial degree bound of $\LC_{x_i}(A)$ and $\LC_{x_i}(B)$ is $\leq d$, so the partial degree bound of $\Gamma_i$ is $\leq 2d^2+2d$.
By Lemma \ref{lm-1}, $\|A/G\|_{\infty},\|B/G\|_{\infty}\leq e^{nd}C$.
By Lemma \ref{lm-5}, $\|R\|_{\infty}\leq (d+1)^{2nd-d}e^{2nd^2}C^{2d}$.
So the absolute value of the leading coefficient of $\Gamma_i$ is $\leq (d+1)^{2nd-d}e^{2nd^2}C^{2d+2}$.
\end{proof}\qed

\subsection{Computing the partial degrees of the GCD}\label{sec-4-2}
Algorithm \ref{alg-41} is used to find the partial degrees of $\gcd(A,B)$.
The main idea is as follows: for each $x_i$, randomly pick a prime $p$ and $\mathbf{a}\in \F_p^n$, and then compute the degree of $\gcd(\Phi_p(A(\Breve{\mathbf{a}},x_i)),\Phi_p(B(\Breve{\mathbf{a}},x_i))$ which equals $\deg_{x_i}G$ with high probability. The following theorem shows how to pick $p$ and ${\mathbf{a}}$ to maintain a high success rate.

\begin{theorem}\label{the-3}
Let $A,B,G\in \Z[x_1,\dots,x_n]$ be polynomials and $G=\gcd(A,B)$. Let $d_i=\deg_{x_i}G$ and $\|A\|_{\infty},\|B\|_{\infty}\leq C$. Let $N=\max\{ 12\lceil(2n^2d-nd)\ln(d+1)+2n^2d^2+(2d+2)n\ln C\rceil,20n^2d^2+20n^2d\}$. If we randomly choose a prime $p$ in $[N+1,2N]$ and randomly choose a vector $\mathbf{a}=(a_1,\dots,a_{n})\in \F_p^{n}$ then
$${\rm Prob}[d_i=\gcd(\Phi_p(A(\Breve{\mathbf{a}},x_i)),\Phi_p(B(\Breve{\mathbf{a}},x_i))),i=1,\dots,n]\geq \frac{3}{4}.$$
\end{theorem}
\begin{proof}
By Lemma \ref{lm-16}, for each $i=1,\dots,n$, there is a non-zero polynomial $\Gamma_i$ for $A$ and $B$ so that if $\Gamma_i(\mathbf{a}) ~{\rm mod}~ p\neq 0$, then  $d_i=\gcd(\Phi_p(A(\Breve{\mathbf{a}},x_i)),\Phi_p(B(\Breve{\mathbf{a}},x_i)))$. Let $\Gamma:=\prod_{i=1}^n\Gamma_i$. Therefore, it suffices to show that $\Gamma(\mathbf{a}) \mod p\neq 0$ with probability  $\geq \frac{3}{4}$.
As $|\LC(\Gamma_i)|\leq (d+1)^{2nd-d}e^{2nd^2}C^{2d+2}$, then $|\LC(\Gamma)|\leq (d+1)^{2n^2d-nd}e^{2n^2d^2}C^{2nd+2n}:=\kappa$. Thus, there are at most $\ln \kappa/\ln N$ different primes in $[N+1,2N]$ vanishing the leading coefficient.
By Lemma \ref{lm-3}, there are at least $\frac{1}{2}N/\ln N$ primes in $[N+1,2N]$.
Therefore, when we randomly choose a prime $p$ in $[N+1,2N]$, the probability of $p$ dividing the leading coefficient of $\Gamma$ is $\leq \frac{\ln \kappa/\ln N}{\frac{1}{2}N/\ln N}=\frac{\ln \kappa}{\frac{1}{2}N}\leq \left((2n^2d-nd)\ln(d+1)+2n^2d^2+(2d+2)n\ln C\right)/ \frac{1}{2}N\leq\frac16$.
Thus with probability $\geq \frac56$, $\Phi_p(\Gamma)\neq0$. Assume $\Phi_p(\Gamma)\neq0$, as $\deg \Gamma\leq 2n^2d^2+2n^2d$ and $p>N\geq 20n^2d^2+20n^2d$, by Lemma \ref{lm-6}, $\Phi_p(\Gamma(\mathbf{a}))\neq0$ with probability $\geq 1-\frac{\deg \Gamma}{N}\geq 1-\frac{2n^2d^2+2n^2d}{20n^2d^2+20n^2d}=\frac{9}{10}$.
So the probability that for all $i=1,2,\dots,n$, $d_i=\gcd(\Phi_p(A(\Breve{\mathbf{a}},x_i)),\Phi_p(B(\Breve{\mathbf{a}},x_i)))$ is $\geq \frac{9}{10}\cdot \frac{5}{6}=\frac{3}{4}$.

\end{proof}\qed

\begin{algorithm}[H]
\caption{Computing partial degrees of the GCD.}\label{alg-41}
\begin{algorithmic}[1]
\REQUIRE $A,B\in \Z[x_1,\dots,x_n]$; a tolerance $0<\varepsilon<1$.
\ENSURE The partial degrees $\deg_{x_i}\gcd(A,B),i=1,\dots,n$ with probability $\geq 1-\varepsilon$.

\STATE Find the partial degree bound $d:=\max_{i=1}^n \max(\deg_{x_i}A,\deg_{x_i}B)$ and coefficient bound $C:=\max(\|A\|_\infty,  \|B\|_\infty)$.

\STATE Let $N:=\max\{ 12
\lceil (2n^2d-nd)\ln(d+1)+2n^2d^2+(2nd+2n)\ln C\rceil,20n^2d^2+20n^2d\}$ and $K:=\lceil \frac{\log_2 \frac{1}{\varepsilon}}{2}\rceil$.\label{alg-41-2}\

\STATE Let $d_i:=d,i=1,2,\dots,n$.\label{alg-41-3}

\FOR{$\ell=1,\dots, K$\label{alg-41-4}}
\STATE Randomly choose a prime $p$ in $[N+1,2N]$ and randomly choose $\mathbf{a}=(a_1,\dots,a_{n})\in \F^{n}_p$.\label{alg-41-5}

\FOR{$i=1,2\dots,n$\label{alg-41-6}}
\IF{$\LC_{x_i}(A)(\mathbf{a}) \mod p=0$ or $\LC_{x_i}(B)(\mathbf{a})\mod p=0$}
\STATE break.
\ENDIF

\STATE Compute the polynomials $u_i:=A(\Breve{\mathbf{a}},x_i)\mod p$ and $v_i:=B(\Breve{\mathbf{a}},x_i)\mod p$ and $g_i:=\gcd(u_i,v_i)$ in $\F_p[x_i]$.\label{alg-41-10}

\IF{$d_i>\deg_{x_i}g_i$}
\STATE $d_i:=\deg_{x_i}g_i$\label{alg-41-12}
\ENDIF

\ENDFOR \label{alg-41-14}
\ENDFOR \label{alg-41-15}
\RETURN $d_i,i=1,\dots,n$.
\end{algorithmic}
\end{algorithm}

\begin{theorem}\label{the-16}
Algorithm \ref{alg-41} works correctly as specified.
\end{theorem}
\begin{proof}
In Step \ref{alg-41-3}, we start with $d_i:=d$ because $d$ is the upper bound of the partial degrees.
For each $\ell$ in Step \ref{alg-41-4}, by Theorem \ref{the-3}, the probability of $d_i=\deg_{x_i}\gcd(A,B),i=1,\dots,n$ in Step \ref{alg-41-12} is $\geq \frac34$. The algorithm computes $K$ times in total, so the probability of correctly computing  $d_i,i=1,\dots,n$ is $\geq 1-(\frac14)^K\geq 1-(\frac14)^{\frac{\log_2 \frac{1}{\varepsilon}}{2}}=1-\varepsilon$. We proved it.
\end{proof}\qed

\begin{theorem}\label{the-17}
Algorithm \ref{alg-41} costs $O^\sim(nt\log C\log\frac{1}{\varepsilon}+n^2t\log^2 d\log\frac{1}{\varepsilon}\log\log C+nd\log\frac{1}{\varepsilon}\log\log C)$ bit operations, where $t=\#A+\#B$.
\end{theorem}
\begin{proof}
We analyse the complexity.
In Step \ref{alg-41-5}, choosing a prime $p$ in $(N,2N]$ requires $O(\log^3N)$ bit operations. Choosing $\mathbf{a}$ costs $O(n\log N)$ bit operations.
In Step \ref{alg-41-10}, computing $u_i$ and $v_i$ costs $O^\sim(nt\log d\log p+t\log C)$
bit operations. Computing $\gcd(u_i,v_i)$ costs $O^\sim(d\log p)$ bit operations. Thus the complexity of Step \ref{alg-41-5}-\ref{alg-41-14} is $O^\sim(n^2t\log d\log p+nt\log C+nd\log p)$ bit operations.
As we do Steps \ref{alg-41-4}-\ref{alg-41-15} $K$ times, the complexity is $O^\sim(n^2tK\log d\log p+ntK\log C+ndK\log p)$ bit operations, that  is, $O^\sim(nt\log C\log \frac{1}{\varepsilon}+n^2t\log^2 d\log\log C\log\frac{1}{\varepsilon}+nd\log\log C\log\frac{1}{\varepsilon})$ bit operations.
\end{proof}\qed

\subsection{An algorithm to verify whether $H\overset{\text{?}}=\gcd(A,B)$}\label{section-4-1}
Although testing $H|A$ and $H|B$ is very fast in practical computation, there seems to be no deterministic algorithm with sparse complexity.
In this section, we first give a randomized technique with sparse complexity to do this.
Then, combined with Algorithm \ref{alg-41}, we give a randomized algorithm to verify whether candidate $H$ is the GCD of $A,B$.

To test $H|A$, the main idea is to reduce it to test $\Phi_p(H(\Breve{\mathbf{a}},x_{i_0})) | \Phi_p(A(\Breve{\mathbf{a}},x_{i_0}))$ for some $i_0\in\{1,2,\dots,n\}$, some prime $p$ and some $\mathbf{a}\in\Z_p^n$, where
$\Phi_p(H(\Breve{\mathbf{a}},x_{i_0}))$, $\Phi_p(A(\Breve{\mathbf{a}},x_{i_0}))$ (cf. Eq. (\ref{eq-2}))
are univariate polynomials in $\Z_p[x_{i_0}]$.
We have the following theorem to ensure the success rate of reduction.

\begin{theorem}\label{the-4}
Let $H,A\in \Z[x_1,\dots,x_n]$ be primitive polynomials and $H\nmid A$.  Let $d$ be the partial degree bound of $H,A$ and $\|H\|_{\infty},\|A\|_{\infty}\leq e^{nd}C$. Let $N=\max\{ 12\lceil(2nd-d)\ln(d+1)+4nd^2+2nd+(2d+2)\ln C\rceil,20nd^2+20nd\}$. If a prime $p$ is chosen at random from $[N+1,2N]$ and we randomly choose a vector $\mathbf{a}=(a_1,\dots,a_{n})\in \F_p^{n}$, then, with probability $\geq \frac{3}{4}$, there exists $i_0\in \{1,\dots,n\}$ so that $\Phi_p(H(\Breve{\mathbf{a}},x_{i_0}))$ does not divide $\Phi_p(A(\Breve{\mathbf{a}},x_{i_0}))$.
\end{theorem}

\begin{proof}
Denote $g:=\gcd(H,A)$. Towards a contradiction assume $\deg_{x_i}g=\deg_{x_i}H,i=1,\dots,n$. Then $g=H$ as $H$ is primitive. So $H|A$, this contradicts $H\nmid A$. So there exists $i_0\in \{1,2,\dots,n\}$ so that $\deg_{x_{i_0}}g<\deg_{x_{i_0}}H$. By Lemma \ref{lm-16}, there exists a non-zero polynomial $\Gamma$
so that if $\Gamma(\mathbf{a}) \mod p\neq 0$, then $$\Phi_p(g(\Breve{\mathbf{a}},x_{i_0})) \sim\gcd(\Phi_p(H(\Breve{\mathbf{a}},x_{i_0}),\Phi_p(A(\Breve{\mathbf{a}},x_{i_0})))$$ and $\deg_{x_{i_0}}\Phi_p(g(\Breve{\mathbf{a}},x_{i_0}))=\deg_{x_{i_0}}g<\deg_{x_{i_0}}\Phi_p(H(\Breve{\mathbf{a}},x_{i_0}))$. Thus, $\Phi_p(H(\Breve{\mathbf{a}},x_{i_0})) \nmid \Phi_p(A(\Breve{\mathbf{a}},x_{i_0}))$.

It suffices to show that $\Gamma(\mathbf{a}) \mod p\neq 0$ with probability  $\geq \frac{3}{4}$. By Lemma \ref{lm-16}, the partial degree bound of $\Gamma$ is $\leq 2d^2+2d$.
As $\|H\|_{\infty},\|A\|_{\infty}\leq e^{nd}C$, by Lemma \ref{lm-16}, $|\LC(\Gamma)|\leq (d+1)^{2nd-d}e^{4nd^2+2nd}C^{2d+2}:=\kappa$.
Thus there are at most $\frac{\ln \kappa}{\ln N}$ different unfavorable primes in $[N+1,2N]$ causing the leading coefficient to vanish.
By Lemma \ref{lm-3}, there are at least $\frac{1}{2}N/\ln N$ primes in $[N+1,2N]$.
Therefore, if we randomly choose a prime $p$ in $[N+1,2N]$, the probability of $p$ dividing the leading coefficient of $\Gamma$ is $\leq \frac{\ln \kappa/\ln N}{\frac{1}{2}N/\ln N}=\frac{2\ln \kappa}{N}\leq \frac{(2nd-d)\ln(d+1)+4nd^2+2nd+(2d+2)\ln C}{1/2\cdot N}=\frac16$.
Thus with probability $\geq \frac56$, $\Phi_p(\Gamma)\neq0$. Assume $\Phi_p(\Gamma)\neq0$.  Since $\deg \Gamma\leq 2nd^2+2nd$ and $p>N\geq 20nd^2+20nd$, by Lemma \ref{lm-6}, $\Phi_p(\Gamma(\mathbf{a}))\neq0$ with probability $\geq 1-\frac{\deg \Gamma}{N}\geq 1-\frac{2nd^2+2nd}{20nd^2+20nd}=\frac{9}{10}$.
So the probability that $\Phi_p(H(\Breve{\mathbf{a}},x_{i_0}))\nmid\Phi_p(A(\Breve{\mathbf{a}},x_{i_0}))$ is $\geq \frac{9}{10}\cdot \frac{5}{6}=\frac{3}{4}$.
\qed
\end{proof}

\begin{algorithm}[H]
\caption{Verifying  whether  $H\overset{\text{?}}=\gcd(A,B)$.}\label{alg-4}
\begin{algorithmic}[1]
\REQUIRE  $A,B,H\in \Z[x_1,\dots,x_n]$ with $\LC(H)>0$; a tolerance $\varepsilon$.
\ENSURE If $H\neq \gcd(A,B)$, then it returns ``$H\neq \gcd(A,B)$" with probability $\geq 1-\varepsilon$; if $H=\gcd(A,B)$, then it returns ``$H=\gcd(A,B)$" with probability $\geq 1-\varepsilon$.


\STATE Compute the contents and the primitive parts of $H$ and $A,B$, and denote them by $C_{H}:=\Cont(H),C_A:=\Cont(A),C_B:=\Cont(B),P_{H}:=\Prim(H),P_A:=\Prim(A),P_B:=\Prim(B)$.\label{alg-4-2}\

\IF {$C_{H}\neq \gcd(C_A,C_B)$}\label{alg-4-21}
\RETURN ``$H\neq \gcd(A,B)$".
\ENDIF

\STATE Find the partial degree bound $d$ of $P_A,P_B$ and coefficient bound $C\geq \|P_A\|_\infty,  \|P_B\|_\infty$.\

\STATE Let $N=\max\{ 12
\lceil (2nd-d)\ln(d+1)+4nd^2+2nd+(2d+2)\ln C\rceil,20nd^2+20nd\}$. Let $K:=\lceil \frac{\log_2 \frac{1}{\varepsilon}}{2}\rceil$.\label{alg-4-3}\

\FOR{$\ell=1,2,\dots,K$\label{alg-4-4}}
\STATE Randomly choose a prime $p$ in $[N+1,2N]$ and randomly choose $\mathbf{a}=(a_1,\dots,a_{n})\in \F^{n}_p$.\label{alg-4-6}

\FOR{$i=1,2\dots,n$\label{alg-4-7}}
\IF{$\LC(H)(\Breve{\mathbf{a}},x_i)=0$}
\STATE break and goto Step \ref{alg-4-4}
\ENDIF

\STATE Compute the polynomials $u_i=P_{H}(\Breve{\mathbf{a}},x_i)$, $v_i=P_A(\Breve{\mathbf{a}},x_i)$ and  $w_i=P_B(\Breve{\mathbf{a}},x_i)$.\label{alg-4-8}

\IF{$u_i\nmid v_i$ or $u_i\nmid w_i$\label{alg-4-9}}
\RETURN ``$H\neq \gcd(A,B)$.\label{alg-4-10}\label{alg-4-12}
\ENDIF
\ENDFOR \label{alg-4-131}
\ENDFOR \label{alg-4-132}
\STATE Find the partial degrees $d_i=\deg_{x_i}\gcd(A,B),i=1,\dots,n$ by Algorithm \ref{alg-41} with tolerance $\varepsilon$.\label{alg-4-16}\

\IF{$d_i=\deg_{x_i}H,i=1,2,\dots,n$\label{alg-4-14}}
\RETURN ``$H=\gcd(A,B)$".\label{alg-4-17}
\ELSE \RETURN ``$H\neq\gcd(A,B)$"\label{alg-4-18}
\ENDIF

\end{algorithmic}
\end{algorithm}

\begin{theorem}
Algorithm \ref{alg-4} works correctly as specified.
\end{theorem}
\begin{proof}
If $H=\gcd(A,B)$, then in Step \ref{alg-4-21}, $C_{H}=\gcd(C_A,C_B)$. In Step \ref{alg-4-9}, $u_i|v_i$ and $u_i|w_i$. So if $\deg_{x_i}\gcd(P_A,P_B),i=1,\dots,n$ are computed correctly in Step \ref{alg-4-16}, it returns ``$H=\gcd(A,B)$" in Step \ref{alg-4-17}. By Theorem \ref{the-16}, the probability of this case is $\geq 1-\varepsilon$.

If $H\neq \gcd(A,B)$, there are three cases.
{\bf Case} 1: $C_{H}\neq \gcd(C_A,C_B)$. This can be detected in Step \ref{alg-4-21}.
{\bf Case} 2: $C_{H}=\gcd(C_A,C_B)$, but $P_{H}\nmid P_A$ or $P_{H}\nmid P_A$. We analyze the case $P_{H}\nmid P_A$, and the analysis of case $P_{H}\nmid P_B$ is the same.
In Step \ref{alg-4-2}, $P_{H}$ and $P_A$ are primitive.
For each $\ell$ in Step \ref{alg-4-4}, by Theorem \ref{the-4}, $P_{H}\nmid P_A$ is correctly detected with a probability of $\geq \frac34$. As we have tested $K$ times in total, the probability of correctly testing  $P_{H}\nmid P_A$ is $\geq 1-(\frac14)^K\geq 1-(\frac14)^{\frac{\log \frac{1}{\varepsilon}}{2}}=1-\varepsilon$. So it returns ``$H\neq \gcd(A,B)$" in Step \ref{alg-4-10}, with a probability of $\geq 1-\varepsilon$.
{\bf Case} 3: $C_{H}=\gcd(C_A,C_B)$, and $P_{H}| P_A$ and $P_{H}| P_B$, but $H\neq \gcd(A,B)$. So there exists $i_0\in \{1,\dots,n\}$ so that $\deg_{x_{i_0}}{H}<\deg_{x_{i_0}}\gcd(A,B)$.
If $\deg_{x_{i_0}}\gcd(A,B)$ is correctly computed in Step \ref{alg-4-16}, then ``$H\neq \gcd(A,B)$" is returned in Step \ref{alg-4-18}. By  Theorem \ref{the-16}, the probability of this case is $\geq 1-\varepsilon$.
We proved it.\qed
\end{proof}

\begin{theorem}\label{the-6}
Algorithm \ref{alg-4} costs $O^\sim(n^2dt\log C\log \frac{1}{\varepsilon})$ bit operations, where $t=\#A+\#B+\#G$.
\end{theorem}
\begin{proof}
We analyse the complexity.
In Step \ref{alg-4-2}, $\Cont(H)$ is the GCD of all the coefficients in $H$, so the cost is $O^\sim(tnd\log C)$ bit operations. As $\Prim(H)=H/\Cont(H)$, computing $\Prim(H)$ requires $O^\sim(ndt\log C)$ bit operations.
It also costs $O^\sim(tnd\log C)$ bit operations to compute $\Cont(A)$,$\Cont(B)$ and $\Prim(A)$,\\$\Prim(B)$.
In Step \ref{alg-4-6}, choosing a prime $p$ in $(N,2N]$ requires $O(\log^3N)$ bit operations. Choosing $\mathbf{a}$ costs $O(n\log N)$ bit operations.
In Step \ref{alg-4-8}, computing $u_i,v_i$ and $w_i$ costs $O^\sim(nt\log d\log p+ndt\log C)$
bit operations. In Step \ref{alg-4-9}, to test $u_i|v_i$ and $u_i|w_i$, as $d\geq \deg u_i,\deg v_i,\deg w_i$, by univariate polynomial division, it costs $O^\sim(d\log p)$ bit operations \cite[p23]{Bini1994Polynomial}. Thus the complexity of Steps \ref{alg-4-7}-\ref{alg-4-131} is $O^\sim(n^2t\log d\log p+n^2dt\log C+nd\log p)$ bit operations.
As we do Steps \ref{alg-4-4}-\ref{alg-4-132} at most for $K$ times, the complexity is $O^\sim(n^2tK\log d\log p+n^2dtK\log C+ndK\log p)$ bit operations, that  is, $O^\sim(n^2dt\log C\log \frac{1}{\varepsilon})$ bit operations.
In Step \ref{alg-4-16}, by Theorem \ref{the-17}, the complexity is $O^\sim(nt\log C\log \frac{1}{\varepsilon}+n^2t\log^2 d\log\log C\log\frac{1}{\varepsilon}+nd\log\log C\log\frac{1}{\varepsilon})$ bit operations.
We proved it.
\end{proof}

\section{A GCD algorithm without a given term bound}
Let $G=\gcd(A,B)$.  Algorithm \ref{alg-1} requires a term bound $T$ for $\#G$ as input, but in practice, it is difficult to give a tight upper bound of the output polynomial before we know its exact form. In this section, we will remove this requirement.

We first observe that even if $T < \#G$, Algorithm \ref{alg-1} may still output $G$.
If $H$ is the output of Algorithm \ref{alg-1}, and $H\neq G$, then by Theorem \ref{the-15}, $H\nmid A$ or $H\nmid B$ or $\deg_{x_i}H \ne \deg_{x_i}G$ for some $i\in\{1,\dots,n\}$, conditions we can test.
%
%
%
%
We let $T=2,2^2,\dots,2^{\lceil n\log_2(d+1)\rceil}$, and use $T$ as the input in Algorithm \ref{alg-1}.
If $T$ is an upper bound of the number of terms of $\gcd(A,B)$, then with probability $\geq \frac{3}{4}$, Algorithm \ref{alg-1} returns the correct GCD.
To improve the success rate,  repeat Algorithm \ref{alg-1} $m$ times, and then at least one of the outputs is correct, with probability $1-4^{-m}$.
To test whether the output is correct, we call Algorithm \ref{alg-4}.

\subsection{Algorithm}

\renewcommand{\algorithmicuntil}{\textbf{end repeat}}

\begin{algorithm}[H]
\caption{GCD algorithm for $\Z[x_1,\dots,x_n]$.}\label{alg-5}
\begin{algorithmic}[1]
\REQUIRE Two polynomials $A,B\in \Z[x_1,\dots,x_n]$ and a tolerance $\varepsilon$.
\ENSURE $G=\gcd(A,B)$ with probability $\geq 1-\varepsilon$.

\STATE Compute the monomial contents of $A,B$ and the monomial primitive parts of $A,B$. Denote them $C_A:=\MoCont(A)$ and $C_B:=\MoCont(B)$.  Set $A:=\MoPrim(A)$ and $B:=\MoPrim(B)$. (For the convenience of description, we still denote the monomial primitive parts as $A,B$.)\label{alg-5-1}

\STATE Compute the individual degree bound $d:=\max_{i=1}^n \min(\deg_{x_i}A,\deg_{x_i}B)$.

\STATE Let $\varepsilon'=\min\{\frac14,\varepsilon\}$, $m=\lceil\frac{1}{2}\log_2 \frac{2}{\varepsilon'} \rceil$, $\kappa=\lceil n\log_2(d+1)\rceil$ and $\eta=\frac{\varepsilon'}{2\kappa\cdot m}.$\

\STATE $T:=2$;
\REPEAT \label{alg-5-3}

\FOR{$i=1,2,\dots,m$\label{alg-5-4}}

\STATE Compute the GCD of $A,B$ with the guess terms bound $T$ using Algorithm \ref{alg-1}.\label{alg-5-6}

 \IF{Algorithm \ref{alg-1} returns ``Failure"}
 \STATE goto Step \ref{alg-5-4}
 \ELSE
 \STATE Let $H$ be the output of Algorithm \ref{alg-1}.
 \ENDIF
\STATE Test if $H=\gcd(A,B)$ by using Algorithm \ref{alg-4} with tolerance $\eta$.\label{alg-5-12}
\IF{ Algorithm \ref{alg-4} outputs ``$H=\gcd(A,B)$''}
\RETURN $\gcd(C_A,C_B) \cdot H$.\label{alg-5-13}
\ENDIF
\ENDFOR

\IF{$T<(d+1)^n$}
\STATE $T:=2T$
\ENDIF
\UNTIL \label{alg-5-16}

\end{algorithmic}
\end{algorithm}

\begin{theorem}
\label{th-10}
Algorithm \ref{alg-5} works correctly as specified.
\end{theorem}
\begin{proof}
 Suppose $t:=\#G$ and set $L:=\lceil \log_2 t\rceil$.
 Then $t\leq 2^{L}< 2t$.
Consider the following two events:
\begin{itemize}
\item[A1:] when $T=2,\dots,2^{L}$, Step \ref{alg-5-12} always tests if
$H=\gcd(A,B)$ or not correctly by using Algorithm \ref{alg-4}.
\item[A2:] when $T=2^{L}$, at least one of outputs of Algorithm \ref{alg-1} in Step \ref{alg-5-6} is correct.
\end{itemize}

Consider the probabilities of the two events. $P(A1) \geq (1-\eta)^{m\cdot L}\geq 1-\eta \cdot m\cdot L\geq 1-\frac{\varepsilon'}{2\cdot\kappa\cdot m}\cdot m\lceil \log_2 t\rceil=1-\frac{\varepsilon'}{2}\cdot \frac{\lceil \log_2 t\rceil}{\lceil n \log_2(d+1)\rceil}.$
Since $t\leq (d+1)^n$, $\lceil \log_2 t\rceil\leq \lceil n\log_2 (d+1)\rceil$. Thus $P(A1)\geq 1-\frac{\varepsilon'}{2}$.
$P(A2)\geq 1-(\frac{1}{4})^{m}\geq 1-(\frac{1}{4})^{\frac12\log_2\frac{2}{\varepsilon'}}=1-\frac{\varepsilon'}{2}$.
So the probability of returning the GCD is $\geq P(A1 \bigcap A2)=P(A2)P(A1|A2)\geq (1-\frac{\varepsilon'}{2})^2\geq
1-\varepsilon'\geq
1-\varepsilon$.
The theorem is proved.
\end{proof}

The following theorem shows the complexity.

\begin{theorem}
\label{th-11}
Let $A,B\in \Z[x_1,\dots,x_n]$ with partial degree bound $d$.
Then Algorithm \ref{alg-5} computes the correct GCD $G=\gcd(A,B)$ using expected $O^\sim(n^2T_{\out}^2d^2$\\$\log^2 \frac{1}{\varepsilon}\log C+n^2dT_{\inp}\log^2 T_{\out}\log^2\frac{1}{\varepsilon}\log C)$ bit operations where $T_{\inp}=\#A+\#B$ and $T_{\out}=\#G$.
\end{theorem}
\begin{proof}
In Step \ref{alg-5-1}, finding the monomial contents and monomial primitive parts of $A,B$ costs $O^\sim(nT_{\inp}\log d+T_{\inp}\log C)$ bit operations.
Let's first analyze the complexity when $T$ is fixed. In Step \ref{alg-5-6},
the complexity is $C_{T,1}:=O^\sim(n^2T^2d^2\log C\log\frac{1}{\varepsilon}+n^2dT_{\inp}\log T\log C\log\frac{1}{\varepsilon})$ bit operations by Theorem \ref{the-7}. In Step \ref{alg-5-12}, the complexity is $C_{T,2}:=O^\sim(n^2d(T+T_{\inp})\log C\log^2\frac{1}{\varepsilon})$ bit operations by Theorem \ref{the-6}.
 %
%
Now let $T=2,2^2,\dots,2^{\kappa-1},2^{\kappa},2^{\kappa},2^{\kappa},\dots$. We keep doing Steps \ref{alg-5-3}-\ref{alg-5-16} until Step \ref{alg-5-13}
outputs a polynomial. However, when $T$ reaches $2^{\kappa}$, it will not be doubled next time since  $\#G\leq (d+1)^n\leq 2^{\kappa}$.
When $T\geq \#G$, in Step \ref{alg-5-6}, at least one of the $m$ outputs of Algorithm \ref{alg-1} is correct with probability $\geq1-\frac{\varepsilon'}{2}$. Set the event
$\mathrm{E}_{\ell}:=\{ {\rm when\ Step\ }$ \ref{alg-5-3}-\ref{alg-5-16} ${\rm run\ to\  the\ } \ell {\rm th\  time,\ the\ algorithm\ stops}\}$,
where $\ell\geq L$.
When event $\mathrm{E}_{\ell}$ occurs, it means that in the $L,L+1,\dots,(\ell-1)$-th loops of Step \ref{alg-5-3}-\ref{alg-5-16}, Step \ref{alg-5-6} computed the wrong GCD or Step \ref{alg-5-12} reached the wrong conclusion. So the probability $P(\mathrm{E}_{\ell})\leq
(\frac{\varepsilon'}{2})^{\ell-L}$. In this case,
$C_{\ell}:=\sum_{i=1}^{\ell}(C_{2^i,1}+C_{2^i,2})\in O^\sim(n^2(2^{\ell})^2d^2\log C\log\frac{1}{\varepsilon}+n^2dT_{\inp}{\ell}^2\log C\log\frac{1}{\varepsilon}+n^2d(2^{\ell}+\ell T_{\inp})\log C\log^2\frac{1}{\varepsilon})$ is the bit complexity.

So the expected complexity is
$
\leq
\sum_{\ell=L}^{\infty}P(\mathrm{E}_\ell) C_{\ell}+C_{L}
$.
%
%
As $\Lambda=1,2$ and
$\varepsilon' 2^{\Lambda-1}\leq \frac12$,  $\sum_{\ell=L}^{\infty} P(\mathrm{E}_{\ell}) (2^{\ell})^{\Lambda}\leq \sum_{\ell=L}^{\infty}(\frac{\varepsilon'}{2})^{\ell-L} (2^{\ell})^{\Lambda} \leq (2^{\Lambda})^{L}\sum_{i=0}^{\infty}(\frac{\varepsilon'}{2})^i(2^{\Lambda})^i\leq T_{\out}^{\Lambda}\frac{1}{1-\varepsilon' 2^{\Lambda-1}}$, the expected bit complexity is
$O^\sim(n^2T_{\out}^2d^2\log C\log\frac{1}{\varepsilon}+n^2d(T_{\out}+T_{\inp}\log^2 T_{\out})\log C\log^2\frac{1}{\varepsilon})$,
that is, $O^\sim(n^2T_{\out}^2d^2\log^2 \frac{1}{\varepsilon}\log C+n^2dT_{\inp}\log^2 T_{\out}\log C\\ \log^2\frac{1}{\varepsilon})$ bit operations.
\end{proof}

\section{Experimental results}\label{sec-exp}
In this section, we compare our algorithm with the default algorithm in Maple 2021.
%
The data are collected on a Windows desktop computer with a
3.00GHz Core i7 processor and 8GB RAM memory.
The codes can be found in
https://github.com/huangqiaolong/Maple-codes-GCD-integers-coeffs.

We have implemented Algorithm \ref{alg-5} that calls Algorithm \ref{alg-1}. In Step \ref{alg-1-5} of Algorithm  \ref{alg-1}, we use the small primes modular GCD in $\Z[y]$ instead of Algorithm \ref{alg-3} as it is faster than the Hensel lifting method.
%
In Step \ref{alg-1-4} of  Algorithm \ref{alg-1}, we find that the ``bad" points in Lemma \ref{the-13} rarely appear.
In our implementation, we randomly choose different primes $p_i$ of sizes $O(nT\log C)$, which is already big enough such that $p_i$ does not divide the coefficients of $G=\gcd(A,B)$ with high probability, according to our experiments. In case some coefficients of $G$ contain certain $p_i$ or $G_{(\mathbf{s},y)}$ is not separate w.r.t $y$ after the substitution, the algorithm returns ``Failure" and we
will repeat the algorithm with double $T$ (When $T>(d+1)^n$, stop increasing) until a correct GCD is obtained. The procedure will terminate with high probability since $T$ is continuously increasing (until $T>(d+1)^n$), but this will increase the computing time which is shown
as the peaks of the black curves in Figures \ref{figrn}, \ref{figrd}, \ref{figrt}.

We compare with Maple's GCD algorithm
(command {\tt gcd}$(f,g)$). The default algorithm for Maple 2021 uses Zippel's algorithm. The implementation is described in \cite{KleineMW05} and is called the LINZIP algorithm.
%
%
We use the Maple command {\tt randpoly}
to construct five groups of random polynomials $f_1,f_2,f_3\in\Z[x_1,\dots,x_n]$.  Let $A:=f_1f_2$ and $B:=f_1f_3$. We compute the GCD of $A$ and $B$ through the Maple algorithm and our new algorithm, respectively, and record the average computing times. The black line is the time of our algorithm, and the red line is the time of the Maple code.  We compare three benchmarks. Each benchmark changes only one parameter, and the others remain unchanged.

\begin{figure}[!hptb]
\begin{minipage}[t]{0.45\linewidth}
\centering
\includegraphics[scale=0.30]{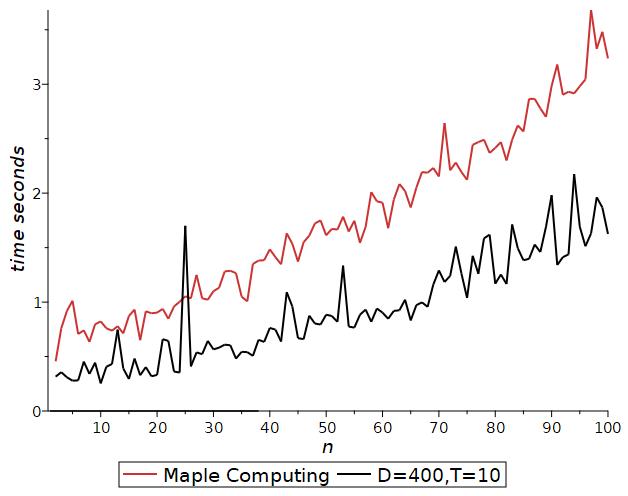}
\caption{Average running time with varying number of variables}
\label{figrn}
\end{minipage}
\hskip5pt
\begin{minipage}[t]{0.45\linewidth}
\centering
\includegraphics[scale=0.3]{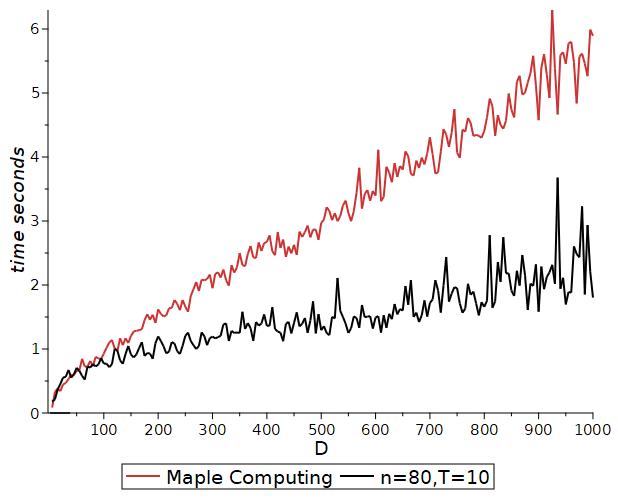}
\caption{Average running time with varying degree} \label{figrd}
\end{minipage}
\end{figure}

\begin{figure}[!hptb]
\begin{minipage}[t]{0.52\linewidth}
\centering
\includegraphics[scale=0.30]{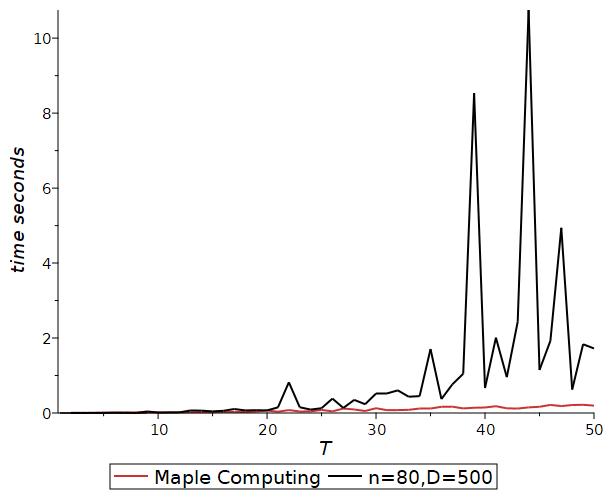}
\caption{Average running time with varying terms}\label{figrt}
\end{minipage}
\end{figure}

{\bf Benchmark 1.}
Generate $f_1,f_2,f_3\in\Z[x_1,\dots,x_n]$
satisfying $\deg f_1=\deg f_2=\deg f_3=400$ and $\#f_1=\#f_2=\#f_3=10$.
Change $n$ from $2$ to $100$. Let $A:=f_1f_2$ and $B:=f_1f_3$.  The average computing times are shown in Figure \ref{figrn}.
From the figure, we can see that the time of our algorithm increases more slowly.

{\bf Benchmark 2.}
Generate  $f_1,f_2,f_3$ satisfying $\#(f_1)=\#(f_2)=\#(f_3)=10$
and $n=80$. Change $D=\deg f_1=\deg f_2=\deg f_3$ from $5$ to $1000$ in step of $5$.
The computing times for $A:=f_1f_2$ and $B:=f_1f_3$ are shown in Figure \ref{figrd}.
It can be seen that our algorithm is faster than Maple's method in most cases.
%
%

{\bf Benchmark 3.}
Generate $f_1,f_2,f_3$ satisfying  $\deg f_1=\deg f_2=\deg f_3=20$ and $n=60$. Change $T=\#(f_1)=\#(f_2)=\#(f_3)$ from $2$ to $50$.
The computing times are shown in Figure \ref{figrt}.
From this figure, we can see that our algorithm runs slower than Maple's algorithm. This is because our GCD algorithm is reduced to the computation of a univariate GCD with degree $O(D T^2)$.  
The cost of Maple's code for a modular $\gcd(A,B)$ in $\F_p[x]$ is $O(\deg(A) \deg(B))$. Maple does not use the fast Euclidean algorithm. Based on this complexity, our cost is approximately $O(D^2T^4)$, which is unfavorable for large $T$.
%

\section{Conclusion and future work}
In this paper, we proposed a new method for computing the GCD of sparse multivariate polynomials  with integer coefficients. 
We map the multivariate polynomials into univariate ones which keep the sparse structure and then recover the target multivariate GCD via factorizations of coefficients.
We also give the explicit bit complexity for the algorithm, which is polynomial in the sparse representation and the partial degree $d$. Our algorithm is better than the previous algorithms in $n$ or $d$.
The algorithm is shown to be faster than the default Maple GCD code if the number of terms in GCD is small.

According to Benchmark 3, our algorithm does not work well when $T$ is large. This is because we should compute a univariate GCD with high degree $O(DT^2)$ and large coefficients. Next, we will look for more effective univariate GCD algorithms for large degrees and coefficients.

\section*{Acknowledgements}
This work was supported by the National Natural Science Foundation of China (Grant No.12571552) and
 NSERC of Canada research grant RGPIN-2019-04441.
\bibliographystyle{abbrv}
\bibliography{mybibfile}

\end{document}